\documentclass[11pt,twoside]{amsart}
\usepackage{latexsym,amssymb,amsmath}
\usepackage{tikz}
\usepackage{multirow}

\numberwithin{equation}{section}
\newtheorem{theorem}{Theorem}[section]
\newtheorem{lemma}[theorem]{Lemma}
\newtheorem{proposition}[theorem]{Proposition}
\newtheorem{corollary}[theorem]{Corollary}

\theoremstyle{definition}
\newtheorem{definition}[theorem]{Definition} 
\newtheorem{procedure}[theorem]{Procedure} 
\newtheorem{remark}[theorem]{Remark}
\newtheorem{example}[theorem]{Example}

\newcommand{\fq}{\mbox{$\mathfrak{q}$}}
\newcommand{\bn}{\mbox{$\mathbb{N}$}}
\newcommand{\bz}{\mbox{$\mathbb{Z}$}}
\newcommand{\bq}{\mbox{$\mathbb{Q}$}}

\newcommand{\bm}{\mbox{$\mathcal{M}$}}
\newcommand{\bl}{\mbox{$\mathcal{L}$}}

\newcommand{\rad}{\mbox{${\rm rad}$}}

\newcommand{\rank}{\mbox{${\rm rank}$}}
\newcommand{\ch}{\mbox{${\rm char}$}}
\newcommand{\diag}{\mbox{${\rm diag}$}}

\newcommand{\adj}{\mbox{${\rm adj}$}}
\newcommand{\hull}{\mbox{${\rm Hull}$}}

\begin{document}


\title[]{Degree and algebraic properties of lattice and matrix ideals}  
\thanks{The second author was partially supported by the
MTM2010-20279-C02-01 grant. The third author was partially supported by SNI}

\author{Liam O'Carroll}
\address{
Maxwell Institute for Mathematical Sciences\\
School of Mathematics\\
University of Edinburgh \\
EH9 3JZ, Edinburgh, Scotland
}
\email{L.O'Carroll@ed.ac.uk}

\author{Francesc Planas-Vilanova}
\address{
Departament de Matem\`atica Aplicada 1\\
Universitat Polit\`ecnica de Catalunya\\
Diagonal 647, ETSEIB \\
08028 Barcelona, Catalunya
}
\email{francesc.planas@upc.edu}

\author{Rafael H. Villarreal}
\address{
Departamento de
Matem\'aticas\\
Centro de Investigaci\'on y de Estudios
Avanzados del
IPN\\
Apartado Postal
14--740 \\
07000 Mexico City, D.F.
}
\email{vila@math.cinvestav.mx}



\begin{abstract} We study the degree
of non-homogeneous lattice ideals over arbitrary fields, and give
formulae to compute the degree in terms of the torsion of certain
factor groups of $\mathbb{Z}^s$ and in terms of relative volumes of
lattice polytopes.  We also study primary decompositions of lattice
ideals over an arbitrary field using the Eisenbud-Sturmfels theory of
binomial ideals over algebraically closed fields. We then use these
results to study certain families of integer matrices (PCB, GPCB, CB,
GCB matrices) and the algebra of their corresponding matrix ideals. In
particular, the family of generalized positive critical binomial
matrices (GPCB matrices) is shown to be closed under transposition,
and previous results for PCB ideals are extended to GPCB ideals.
Then, more particularly, we give some applications to the theory of
$1$-dimensional binomial ideals.  If $G$ is a connected graph, we show
as a further application that the order of its sandpile group is the
degree of the Laplacian ideal and the degree of the toppling ideal. We
also use our earlier results to give a structure theorem for graded
lattice ideals of dimension $1$ in $3$ variables and for homogeneous
lattices in $\mathbb{Z}^3$ in terms of critical binomial ideals (CB
ideals) and critical binomial matrices, respectively, thus
complementing a well-known theorem of Herzog on the toric ideal of a
monomial space curve.
\end{abstract}

\maketitle

\section{Introduction}\label{intro}

Let $S=K[t_1,\ldots,t_s]$ be a polynomial ring over a field $K$. As
usual, $\mathfrak{m}$ will denote the maximal ideal of $S$ generated
by $t_1,\ldots,t_s$. For an arbitrary ideal $I$ of $S$ there are
various ways of introducing the notion of {\it degree\/}; let us
briefly recall one of them. The vector space of polynomials in $S$
(resp. $I$) of degree at most $i$ is denoted by $S_{\leq i}$ (resp.
$I_{\leq i}$). If $H_I^a(i):=\dim_K(S_{\leq i}/I_{\leq i})$ is the
affine Hilbert function of $S/I$ and $k$ is the Krull dimension of
$S/I$, the positive integer
$$
\deg(S/I):=k!\, \lim_{i\rightarrow\infty}{H_I^a(i)}/{i^k}
$$
is called the {\it degree\/} or {\it multiplicity\/} of $S/I$. If
$S=\oplus_{i=0}^{\infty}S_i$ has the standard grading and
$I\subset S$ is a graded ideal, then $H_I^a$ is the Hilbert-Samuel
function of $S/I$ with respect to $\mathfrak{m}$ in the sense of
\cite[Definition~B.3.1]{Vas1}. The notion of degree 
plays a central role in this paper. One of our aims is to 
give a formula for the degree when $I$ is any lattice
ideal.

The set of nonnegative integers (resp. positive integers) is denoted
by $\mathbb{N}$ (resp. $\mathbb{N}_+$). A {\it binomial\/} is a
polynomial of the form $t^b-t^c$, where $b,c\in \mathbb{N}^s$ and
where, if $c=(c_i)\in\mathbb{N}^s$, we set $t^c=t_1^{c_1}\cdots
t_s^{c_s}$. We use the term ``binomial'' as a shorthand for what
elsewhere has been called {\it pure difference binomial\/}
\cite[p.~2]{EisStu} or {\it unital binomial\/} \cite{kahle-miller}.  A
     {\it binomial ideal\/} is an ideal generated by binomials.  The
     set $\{i\, |\, c_i\neq 0\}$, denoted by ${\rm supp}(c)$, is
     called the {\it support\/} of $c$.

Consider an $s\times m$ integer matrix $L$ with column vectors
$a_1,\ldots,a_m$. Each $a_i$ can be written uniquely as
$a_i=a_i^+-a_i^-$, where $a_i^+$ and $a_i^-$ are in $\mathbb{N}^s$ and
have disjoint support.  The {\it matrix ideal\/} of $L$, denoted by
$I(L)$, is the ideal of $S$ generated by
$t^{a_1^+}-t^{a_1^-},\ldots,t^{a_m^+}-t^{a_m^-}$.  A {\it matrix
  ideal\/} is an ideal of the form $I(L)$ for some $L$. Matrix ideals
are a special class of binomial ideals. Some of our results concern
certain integer matrices and the algebra of their matrix
ideals.

A subgroup $\mathcal{L}$ of $\mathbb{Z}^s$ is called a {\it
  lattice\/}. A {\it lattice ideal\/}, over the field $K$, is an ideal
of $S$ of the form $I(\mathcal{L})=(t^{a^+}-t^{a^-}\vert\,
a\in\mathcal{L})$ for some lattice $\mathcal{L}$ in
$\mathbb{Z}^s$. Let $\mathcal{L}$ be the lattice generated by the
columns $a_1,\ldots,a_m$ of an integer matrix $L$.  It is well-known
that $I(L)$ and $I(\mathcal{L})$ are related by the equality
$I(\mathcal{L})=(I(L)\colon(t_1\cdots t_s)^\infty)$ and that
$I(\mathcal{L})$ is also a matrix ideal. The class of lattice ideals
has been studied in many places, see \cite{cca} and the references
there.  This concept is a natural generalization of a toric ideal.

Using commutative algebra methods and the Eisenbud-Sturmfels theory of
binomial ideals over algebraically closed fields \cite{EisStu}, in
this paper we study algebraic properties and primary decompositions of
lattice ideals and binomial ideals of a variety of types. By and
large, we focus on the structure of the class of graded binomial
ideals $I$ that satisfy the vanishing condition $V(I,t_i)=\{0\}$ for
all $i$, where $V(\cdot)$ is the variety of $(\cdot)$. This class of
ideals includes the graded lattice ideals of dimension $1$
\cite{ci-lattice}, the vanishing ideals over finite fields of
algebraic toric sets \cite{algcodes}, the toric ideals of monomial
curves \cite{He3}, the Herzog-Northcott ideals \cite{opHN}, the PCB
ideals \cite{opPCB} and the Laplacian ideals of complete graphs
\cite{riemann-roch}. We will also present some other interesting
families of ideals that satisfy this hypothesis. In particular, for
$s=3$, we completely determine the structure of any graded lattice
ideal in terms of critical binomial ideals and also the structure of
any homogeneous lattice in $\mathbb{Z}^3$. The transpose of $L$ is
denoted by $L^\top$. If $I(L^\top)$ is graded, following
\cite{opPCB}, we study when $I(L)$ is also graded.

The contents of this paper are as follows. In Section~\ref{prelim}, we 
introduce the notion of degree via Hilbert polynomials and 
observe that the degree is independent of the base field $K$ 
(Proposition~\ref{nov26-12}). We present some of the results on
lattice ideals that will be needed throughout the paper, introduce
some notation, and recall how the structure of
$T(\mathbb{Z}^s/\mathcal{L})$, the torsion subgroup of
$\mathbb{Z}^s/\mathcal{L}$,  
can be read off from the normal form of $L$. All results of this
section are well-known.  

In Section~\ref{prim-dec-deg-tor}, we study primary decompositions of
lattice ideals. The first main result is an auxiliary theorem that
relates the degree of $S/I(\mathcal{L})$ and the torsion of
$\mathbb{Z}^s/\mathcal{L}$ to the primary decomposition of
$I(\mathcal{L})$ over an arbitrary field $K$
(Theorem~\ref{bounds-for-the-number-of-associated-primes}).  If $K$ is
algebraically closed, the primary decomposition of $I(\mathcal{L})$ is
given in \cite[Corollary~2.5]{EisStu} in terms of lattice ideals of
partial characters of the saturation of $\mathcal{L}$. In this
situation, the primary components are generated by polynomials of the
form $t^a-\lambda t^b$, where $0\neq\lambda\in K$. Let $\gamma$ be the
order of $T(\mathbb{Z}^s/\mathcal{L})$. If $K$ is a field containing
the $\gamma$-th roots of unity with $\ch(K)=0$ or $\ch(K)=p$, $p$ a
prime with $p\nmid \gamma$, it is well-known that $I(\mathcal{L})$ is
a radical ideal (see Theorem~\ref{may24-05}). Assuming that
$I(\mathcal{L})$ is a graded lattice ideal of dimension $1$, following
\cite{opPCB} we give explicitly the minimal primary decomposition of
$I(\mathcal{L})$ in terms of the normal decomposition of $L$
(Theorem~\ref{main-new}).

Section~\ref{degree-section} is devoted to developing a formula for
the degree of any lattice ideal. First, we exhibit a formula for the
degree that holds for any toric ideal
(Theorem~\ref{toric-degree-non-homg}), the graded case was shown in
\cite[Theorem~4.16, p.~36]{Stur1} and \cite{ehrhart}.  If $S$ has the
standard grading and $I(\mathcal{L})$ is a graded lattice ideal of
dimension $1$, the degree of $S/I(\mathcal{L})$ is the order of
$T(\mathbb{Z}^s/\mathcal{L})$ \cite{degree-lattice}. As usual, we
denote the {\it relative volume\/} of a lattice polytope $\mathcal{P}$
by ${\rm vol}(\mathcal{P})$ and the convex hull of a set $\mathcal{A}$
by ${\rm conv}(\mathcal{A})$.

We come to the main result of  Section~\ref{degree-section}.

\noindent {\bf Theorem~\ref{degree-lattice-char=any}}{\it\ {\rm (a)} 
If ${\rm rank}(\mathcal{L})=s$, then
$\deg(S/I(\mathcal{L}))=|\mathbb{Z}^s/\mathcal{L}|$.\vspace{-2mm}
\begin{itemize} 
\item[\rm (b)] If ${\rm rank}(\mathcal{L})<s$, there is an integer
  matrix $A$ of size $(s-r)\times s$ with ${\rm rank}(A)=s-r$ such
  that we have the containment of rank $r$ lattices ${\mathcal
    L}\subset {\rm ker}_\mathbb{Z}(A)$.

\item[\rm (c)] If ${\rm rank}(\mathcal{L})<s$ and $v_1,\ldots,v_s$
are the columns of $A$, then 
$$
\deg(S/I(\mathcal{L}))=
\frac{|T(\mathbb{Z}^s/\mathcal{L})|(s-r)!{\rm
vol}({\rm conv}(0,v_1,\ldots,v_s))}{|T(\mathbb{Z}^{s-r}/\langle 
v_1,\ldots,v_s\rangle)|}.
$$
\end{itemize}}
One can effectively use Theorem~\ref{degree-lattice-char=any} to
compute the degree of any lattice ideal (Examples~\ref{oct10-12} and
\ref{nov25-12}). For certain families, we show explicit formulae for
the degree (Corollary~\ref{nov28-12}).  For a $1$-dimensional lattice
ideal $I(\mathcal{L})$, not necessarily homogeneous, we can express
the degree in terms of a $\mathbb{Z}$-basis of the lattice
$\mathcal{L}$ (Corollary~\ref{inspired-by-francesc-example},
Example~\ref{jan10-13}).

Section~\ref{primary-dec-1} focuses on graded binomial ideals
satisfying the vanishing condition $V(I,t_i)=\{0\}$ for all $i$. For
ideals of 
this type, we characterize when they are lattice ideals
(Proposition~\ref{2variables}). This enables us to present some
applications of the main result of Section~\ref{prim-dec-deg-tor} to
the theory of binomial ideals.

We show the following result on the structure of graded matrix ideals
$I$, writing $\hull(I)$ for the intersection of the isolated primary
components of $I$.

\noindent {\bf Proposition~\ref{structureI(L)}}{\it\ 
Let $I$ be the matrix ideal of an $s\times m$
integer matrix $L$ and let $\bl$ be the lattice spanned by the columns
of $L$. Suppose that $I$ is graded and that $V(I,t_i)=\{0\}$ for
all $i$. Then:
\begin{itemize}
\item[$(a)$] $I$ has a minimal primary decomposition either of the
  form $I=\fq_1\cap\cdots\cap \fq_c$, if $I$ is unmixed, or else
  $I=\fq_1\cap\cdots\cap \fq_c\cap\fq$, if $I$ is not unmixed, where
  the $\fq_i$ are $\mathfrak{p}_i$-primary ideals with ${\rm
    ht}(\mathfrak{p}_i)=s-1$, and $\fq$ is an $\mathfrak{m}$-primary
  ideal.
\item[$(b)$] $I(\bl)=\fq_1\cap\cdots\cap \fq_c=\hull(I)$.  
\item[$(c)$] $\rank(L)=s-1$ and there exists $\mathbf{d}\in\bn^{s}_{+}$ with
  $\mathbf{d} L=0$. 
\item[$(d)$] If $I$ is not unmixed and $h=t_1\cdots t_s$, there exists
  $a\in\bn_+$ such that $I(\bl)=(I\colon h^{a})$, $\fq=I+(h^a)$ is an
  irredundant $\mathfrak{m}$-primary component of $I$ and
  $I=I(\bl)\cap \fq$.
\item[$(e)$] Either $c\leq |T(\mathbb{Z}^s/\mathcal{L})|$, if
  $\ch(K)=0$, or else $c\leq |G|$, if $\ch(K)=p$, $p$ a prime, where
  $G$  is the unique largest subgroup of $T(\mathbb{Z}^s/\mathcal{L})$
  whose order is relatively prime to $p$. If $K$ is algebraically
  closed, then equality holds.
\end{itemize}}

As a consequence, for the family of matrix ideals satisfying the
hypotheses of Proposition~\ref{structureI(L)}, we obtain the following
formula for the degree:
$$
\deg(S/I)=\max\{d_1,\ldots,d_s\}|T(\bz^s/\bl)|,
$$
where $\mathbf{d}=(d_1,\ldots ,d_s)\in\bn^s_+$ is the weight vector, with
$\gcd(\mathbf{d})=1$, that makes
the matrix ideal $I$ homogeneous (Corollary~\ref{degreeI(L)}). 

In Section~\ref{gpcb-section} we restrict our study to matrix ideals 
associated to square integer matrices of a certain type. Throughout,
set ${\bf 1}=(1,\ldots ,1)$.   

\begin{definition}\label{square-matrix-new}
Let $a_{i,j}\in\mathbb{N}$, $i,j=1,\ldots,s$, and let $L$ be an
$s\times s$ matrix of the
following special form:
\begin{eqnarray}\label{gcb-matrix}
L=\left(\begin{array}{rrrr}
a_{1,1}&-a_{1,2}&\cdots&-a_{1,s}\\
-a_{2,1}&a_{2,2}&\cdots&-a_{2,s}\\
\vdots\phantom{+}&\vdots\phantom{+}&\cdots&\vdots\phantom{+}\\
-a_{s,1}&-a_{s,2}&\cdots&a_{s,s}
\end{array}\right).
\end{eqnarray}
The matrix $L$ is called: (a) a {\em pure binomial\/} matrix (PB
matrix, for short) if $a_{j,j}>0$ for all $j$, and for each column of
$L$ at least one off-diagonal entry is non-zero; (b) a {\em positive
  pure binomial} matrix (PPB matrix) if all the entries of $L$ are
non-zero; (c) a {\em critical binomial} matrix (CB matrix) if $L$ is a
PB matrix and $L{\bf 1}^{\top}=0$; (d) a {\em positive critical
  binomial} matrix (PCB matrix \cite{opPCB}) if all the entries of $L$
are non-zero and $L{\bf 1}^{\top}=0$; (e) a {\em generalized critical
  binomial} matrix (GCB matrix), if $L$ is a PB matrix and there
exists ${\bf b}\in \mathbb{N}^{s}_{+}$ such that $L\mathbf{b}^{\top
}=0$; and (f) a {\em generalized positive critical binomial} matrix
(GPCB matrix) if all the entries of $L$ are non-zero and there exists
${\bf b}\in \mathbb{N}^{s}_{+}$ such that $L\mathbf{b}^{\top }=0$.

If $L$ is a PB matrix, we will call $I(L)=(f_1,\ldots ,f_s)$ the PB
{\it ideal associated\/} to $L$, where $f_i$ is the binomial defined
by the $i$-th column of $L$ (see Definition~\ref{matrix-ideal}). We
will use similar terminology when $L$ is a PPB, CB, PCB, GCB or GPCB
matrix.
\end{definition}

Summarizing, we have the following inclusions among 
these classes of matrices and ideals:
\[
\begin{array}{ccccc}
\mbox{CB}&\subset&\mbox{GCB}&\subset&\mbox{PB}\\ 
\cup & &\cup & & \cup
\\ \mbox{PCB}&\subset&\mbox{GPCB}&\subset&\mbox{PPB}
\end{array}
\]
It turns out that $L$ is a CB matrix if and only if $L$
is the Laplacian matrix of a weighted digraph without sinks or 
sources (see Section~\ref{laplacian-section}). Thus, in principle, we
can and will use the techniques of algebraic graph theory
\cite{Biggs,godsil}, matrix theory \cite{nonnegative-matrices} and
digraph theory \cite{digraphs} to study CB and GCB  
matrices and their matrix ideals.  

The {\it support\/} of a polynomial $f$, denoted by ${\rm supp}(f)$,
is the set of variables that occur in $f$.  If $I(L)=(f_1,\ldots,f_s)$
is the matrix ideal of a PB matrix $L$ and $|{\rm supp}(f_j)|\geq 4$,
for $j=1,\ldots ,s$, we show that $I(L)$ is not a lattice ideal
(Proposition~\ref{oct6-12}).  Let $g_1,\ldots,g_s$ be the binomials
defined by the rows of a GCB matrix $L$ and let $I$ be the matrix
ideal of $L^\top$. If $V(I,t_i)=\{0\}$ and $|{\rm supp}(g_i)|\geq 3$
for all $i$, we show that $I$ is not a complete intersection
(Proposition~\ref{feb9-13}).

The GPCB matrices (resp. ideals) are a generalization of the PCB
matrices (resp. ideals) introduced and studied in \cite{opPCB}. The
origin of this generalization is in the overlap between the results in
\cite[Section~3]{degree-lattice} and the results in
\cite[Section~5]{opPCB} (see \cite[Remark~5.8]{opPCB}). While the
class of PCB matrices is easily seen not to be closed under
transposition, the wider class of GPCB matrices is shown to be closed
under transposition (Theorem~\ref{GPCB-new}).

We are also interested in displaying new families of binomial ideals
verifying the usual hypotheses above, namely, graded matrix ideals $I$
of integer matrices such that $V(I,t_i)=\{0\}$ for all $i$. Some new
such families are the GPCB ideals (Remark~\ref{onecanapply}), the
Laplacian ideals associated to connected weighted graphs
(Proposition~\ref{laplacian-appl}) and the GCB ideals that arise from
matrices with strongly connected underlying digraphs
(Proposition~\ref{gcb-with-strongly-connected-digraph}).  Thus we can
apply to these families some of the results of this article.

In the rest of Section~\ref{gpcb-section}, we extend to GPCB ideals
some properties that hold for PCB ideals.  We give an explicit syzygy
among the generators of a GPCB ideal
(Proposition~\ref{explicitrelation}). This will be used to give an
explicit description of an irredundant embedded component of a GPCB
ideal in at least $4$ variables (Theorem~\ref{embeddedcomp}).  We give
an explicit description of the hull of a GPCB ideal and, if $s\geq 4$,
of an irredundant embedded component (Proposition~\ref{explicitui} and
Theorem~\ref{embeddedcomp}). For $s=2$, we give a description of a
GPCB ideal $I(L)$ and its hull in terms of the entries of the matrix
$L$ (Lemma~\ref{cases=2}). This description will be used in
Section~\ref{lattice-dim1-3vars} (see
Proposition~\ref{when-is-unmixed}) to characterize when $I(L)$ is a
lattice ideal.

In Section~\ref{laplacian-section} we show how our results apply to
matrix ideals arising from Laplacian matrices of weighted graphs and
digraphs. We are interested in relating the combinatorics of a graph
(resp. digraph) to the algebraic invariants and properties of the
matrix ideal associated to the Laplacian matrix of the graph
(resp. digraph).  Let $G =(V,E,w)$ be a weighted simple graph, where
$V=\{t_1,\ldots,t_s\}$ is the set of vertices, $E$ is the set of edges
and $w$ is a weight function that associates a weight $w_e$ with every
$e\in E$. The Laplacian matrix of $G$, denoted by $L(G)$, is a prime
example of a CB matrix. Laplacian matrices of complete graphs are PCB
matrices; this type of matrix occurs in \cite{riemann-roch}.  The
matrix ideal $I\subset S$ of $L(G)$ is called the {\it Laplacian
  ideal} of $G$.  If $I\subset S$ is the Laplacian ideal of $G$, the
lattice ideal $I(\mathcal{L})=(I\colon(t_1\cdots t_s)^\infty)$ is
called the {\it toppling\/} ideal of $G$
\cite{riemann-roch,perkinson}. If $G$ is connected, the toppling ideal
has dimension $1$.

The torsion subgroup of the factor group $\mathbb{Z}^s/{\rm
  Im}(L(G))$, denoted by $K(G)$, is called the {\it critical group\/}
or the {\it sandpile group} of $G$ \cite{alfaro-valencia,lorenzini}.
The group $K(G)$ is equal to the torsion subgroup of
$\mathbb{Z}^s/\mathcal{L}$.  Below, we denote the set of edges of $G$
incident to $t_i$ by $E(t_i)$. Thus we can give an application of our
earlier results to this setting.

\noindent {\bf Proposition~\ref{laplacian-appl}}{\it\ Let $G=(V,E,w)$
be a connected weighted simple graph with 
vertices $t_1,\ldots,t_s$ and let $I\subset S$ be its
Laplacian ideal. Then the following hold.
\begin{itemize}
\item[(a)] $V(I,t_i)=\{0\}$ for all $i$.
\item[(b)] $\deg(S/I)=\deg(S/I(\mathcal{L}))=|K(G)|$.  
\item[(c)] ${\rm Hull}(I)=I(\mathcal{L})$.
\item[(d)] If $|E(t_i)|\geq 3$ for all $i$, then $I$ is 
not a lattice ideal. 
\item[(e)] If $G=\mathcal{K}_s$ is a complete graph, then
$\deg(S/I)=s^{s-2}$.   
\end{itemize}}

As another application, we show that the Laplacian ideal is an almost
complete intersection for any connected simple graph without vertices
of degree $1$ (Proposition~\ref{aci-laplacian}).

Given a square integer matrix $L$, we denote its underlying digraph
by $G_L$ (Definition~\ref{underlying-digraph}). If $L$ is a GCB matrix
and $G_L$ is strongly connected, we show that $L^\top$ is a GCB
matrix (Theorem~\ref{GCB-new-perron-Frobenius}).
If $L$ is a GCB matrix, we show that $G_L$ is 
strongly connected if and only if $V(I(L),t_i)=\{0\}$ for
all $i$ (Proposition~\ref{gcb-with-strongly-connected-digraph}). 
Thus, the results of the previous sections can also be applied to GCB 
ideals that arise from matrices with strongly connected underlying
digraphs.

Finally, in Section~\ref{lattice-dim1-3vars} we focus on matrix ideals
with $s=3$ and apply our earlier results in this setting. From
\cite[Theorem~6.1]{peeva-sturmfels}, it follows that graded lattice
ideals of height $2$ in $3$ variables are generated by at most $3$
binomials. The main results of Section~\ref{lattice-dim1-3vars}
uncover the structure of this type of ideal and the structure of
graded lattices of rank $2$ in $\mathbb{Z}^3$. We show that a graded
lattice ideal in $K[t_1,t_2,t_3]$ of height $2$ is generated by a full
set of critical binomials (Definition~\ref{fullset},
Theorem~\ref{nov18-12}); our proof follows that of
\cite[pp.~137--140]{kunz}. This result complements the well-known
result of Herzog \cite{He3} showing that the toric ideal of a monomial
space curve is generated by a full set of critical binomials.

It is easy to see that an ideal $I\subset K[t_1,t_2]$ is a graded
lattice ideal of dimension $1$ if and only if $I$ is a PCB ideal. The
main result of Section~\ref{lattice-dim1-3vars} is the analogue of
this result in the case of $3$ variables. Concretely, an ideal
$I\subset K[t_1,t_2,t_3]$ is a graded lattice ideal of dimension $1$
if and only if $I$ is a CB ideal (Theorem~\ref{mainSection}). Then we
show that the graded lattices of rank $2$ in $\mathbb{Z}^3$ are
precisely the lattices generated by the columns of a CB matrix of size
$3$ (Corollary~\ref{jan6-13}).  As a corollary of
Theorem~\ref{mainSection}, and for $s=3$, we deduce a characterization
of the structure of the hull of a GCB ideal
(Corollary~\ref{hull(GPCB)}).

For all unexplained
terminology and additional information,  we refer to 
\cite{EisStu,cca,monalg} (for the theory of binomial and lattice
ideals), \cite{CLO,singular-book} (for Gr\"obner bases and Hilbert
functions) and \cite{BHer,Mats,ZS} (for commutative algebra).

\section{Preliminaries}\label{prelim}

In this section, we 
present some of the results that will be needed throughout the paper
and introduce some more notation. All results of this
section are well-known. To avoid repetitions, we continue to employ
the notations and 
definitions used in Section~\ref{intro}.

Let $S=K[t_1,\ldots,t_s]$ be a polynomial ring over a field $K$ and
let $I$ be an ideal of $S$. As usual, $\mathfrak{m}$ will denote the 
maximal ideal of $S$ generated by $t_1,\ldots,t_s$. 
The vector space of polynomials in $S$
(resp. $I$) of degree at most $i$ is denoted by $S_{\leq i}$ (resp.
$I_{\leq i}$). The functions 
$$
H_I^a(i)=\dim_K(S_{\leq i}/I_{\leq i})\ \ \mbox{ and }\ \
H_I(i)=H_I^a(i)-H_I^a(i-1)
$$
are called the {\it affine Hilbert function} and the {\it Hilbert
function} of $S/I$ respectively. 

Let $k=\dim(S/I)$ be the Krull dimension of $S/I$. According 
to \cite[Remark~5.3.16, p.~330]{singular-book}, there are unique
polynomials $h^a_I(t)=\sum_{j=0}^{k}a_jt^j\in 
\mathbb{Q}[t]$ and $h_I(t)=\sum_{j=0}^{k-1}c_jt^j\in
\mathbb{Q}[t]$ of degrees $k$ and $k-1$, respectively, such that
$h^a_I(i)=H_I^a(i)$ and $h_I(i)=H_I(i)$ for 
$i\gg 0$.  By convention, the zero polynomial has degree $-1$. Notice 
that $a_k(k!)=c_{k-1}((k-1)!)$ for $k\geq 1$. If $k=0$, then
$H_I^a(i)=\dim_K(S/I)$ for $i\gg 0$. 

\begin{definition}
The integer $a_k(k!)$, denoted by ${\rm deg}(S/I)$, is 
called the {\it degree\/} of $S/I$. 
\end{definition}

\begin{remark}
If $S=\oplus_{i=0}^{\infty}S_i$ has the standard grading and
$I\subset S$ is a graded ideal, then $H_I(i)$ is equal to 
$\dim_K(S_i/I_i)$  
for all $i$, and $H_I^a$ is the Hilbert-Samuel
function of $S/I$ with respect to 
$\mathfrak{m}$ in the sense of
\cite[Definition~B.3.1]{Vas1}. 
\end{remark}

We will use the following multi-index notation: for
$a=(a_1,\ldots,a_s)\in\mathbb{Z}^s$, set $t^a=t_1^{a_1}\cdots
t_s^{a_s}$. Notice that $t^a$ is a monomial in the Laurent polynomial
ring  $T:=K[t_1^{\pm 1},\ldots,t_s^{\pm 1}]$. If $a_i\geq 0$ for all $i$,
$t^a$ is just a monomial in $S$.

\begin{definition} The {\it graded reverse lexicographical order\/} 
(GRevLex for short) on the monomials of $S$ is defined as $t^b\succ
  t^a$ if and only if $\deg(t^b)>\deg(t^a)$, or $\deg(t^b)=\deg(t^a)$
  and the last nonzero entry of $b-a$ is negative.
\end{definition}

Let $\succ$ be a monomial order on $S$ and let $I\subset S$ be an
ideal. As usual, if $g$ is a
polynomial of $S$, we will denote the {\it leading term\/} of $g$ 
by ${\rm in}(g)$ and the {\it initial ideal\/} of $I$ by ${\rm
in}(I)$. We refer to \cite{CLO} for the
theory of Gr\"obner bases. Let $u=t_{s+1}$ be a new variable. For
$f\in S$ of degree $e$ define 
$$
f^h=u^ef\left({t_1}/{u},\ldots,{t_s}/{u}\right);
$$
that is,  $f^h$ is the {\it homogenization\/} of the polynomial 
$f$ with respect to $u$. The {\it homogenization\/} 
of $I$ is the ideal $I^h$ of $S[u]$ 
given by $I^h=(f^h|\, f\in I)$, and $S[u]$ is given the standard
grading. 

The Gr\"obner bases of $I$ and $I^h$ are nicely related.

\begin{lemma}\label{elim-ord-hhomog}
Let $I$ be an ideal of $S$ and let $\succ$ be the GRevLex order
on $S$ and $S[u]$ respectively.

{\rm (a)} If $g_1,\ldots,g_r$ is a Gr\"obner basis of $I$, then
$g_1^h,\ldots,g_r^h$ is a Gr\"obner basis of $I^h$. 

{\rm (b)} $H_I^a(i)=H_{I^h}(i)$ for $i\geq 0$.

{\rm (c)} $\deg(S/I)=\deg (S[u]/I^{h})$. 

\end{lemma}

\begin{proof} (a): This follows readily 
from \cite[Propositions~2.4.26 and 2.4.30]{monalg}. (b): Fix $i\geq 0$. The mapping
  $S[u]_{i}\rightarrow S_{\leq i}$ induced by mapping $u\mapsto 1$ is
  a $K$-linear surjection. Consider the induced composite $K$-linear
  surjection $S[u]_{i}\rightarrow S_{\leq i}\rightarrow S_{\leq
    i}/I_{\leq i}.$ An easy check shows that this has kernel
  $I_{i}^{h}$. Hence, we have a $K$-linear isomorphism of
  finite-dimensional $K$-vector spaces
\begin{equation*}
S[u]_{i}/I_{i}^{h}\simeq S_{\leq i}/I_{\leq i}.
\end{equation*}
Thus $H_{I}^{a}(i)=H_{I^{h}}(i)$. (c): From classical theory
\cite[p.~192]{ZS}, $\dim(S[u]/I^h)$ is equal to $\dim(S/I)+1$. Hence 
the equality follows from (b). 
\end{proof}

\begin{proposition}\label{additivity-of-the-degree}
Let $I$ be an ideal of $S$ and let
$\mathfrak{p}_1,\ldots,\mathfrak{p}_r$ be the set of associated
primes of $I$ of dimension $\dim(S/I)$. If
$I=\mathfrak{q}_1\cap\cdots\cap\mathfrak{q}_m$ 
is a minimal primary decomposition of $I$ such that
${\rm rad}(\mathfrak{q}_i)=\mathfrak{p}_i$ for $i=1,\ldots,r$, then 
$\deg(S/I)=\sum_{i=1}^r\deg(S/\mathfrak{q}_i)$.
\end{proposition}

\begin{proof} 
By \cite[p.~181, $\mbox{[7]--[9]}$]{ZS} and \cite[top of p.~192]{ZS},
$\mathfrak{q}_{1}^{h},\ldots,,\mathfrak{q}_{r}^{h}$ are the primary
components of $I^{h}$ of maximal dimension. Hence, by (part of) the
associativity law of multiplicities (cf. \cite[Lemma~5.3.11,
  p.~327]{singular-book}), we get
\begin{equation*}
\textstyle\deg (S[u]/I^{h})=\sum_{i=1}^{r}\deg
(S[u]/\mathfrak{q}_{i}^{h}).
\end{equation*}
Hence, by Lemma~\ref{elim-ord-hhomog}(c), 
$\deg (S/I)=\sum_{i=1}^{r}\deg (S/\mathfrak{q}_{i})$.
\end{proof}

\begin{definition}
The torsion
subgroup of an abelian group 
$(M,+)$, denoted by $T(M)$, is the set of all $x$ in $M$ such that
${\ell}x=0$ for some $\ell\in\mathbb{N}_+$.  
\end{definition}

\subsection*{Binomial and lattice ideals}  

Let $\mathcal{L}\subset\mathbb{Z}^s$ be a lattice and let
$I(\mathcal{L})\subset S$ be its lattice ideal.  It is well-known that
the height of $I(\mathcal{L})$ is the rank of $\mathcal{L}$ and that
$I(\mathcal{L})$ is a toric ideal if and only if
$\mathbb{Z}^s/\mathcal{L}$ is free as a $\mathbb{Z}$-module
\cite[p.~131]{cca}. Let $p$ be the characteristic of the field $K$.
The next result will be relevant in our context because it says that
when $p=0$, or when $p>0$ and $p$ is relatively prime to the
cardinality of the torsion subgroup of $\mathbb{Z}^s/\mathcal{L}$,
then $I(\mathcal{L})$ is a radical ideal and hence $I(\mathcal{L})$
has a unique minimal primary decomposition (see
Theorem~\ref{main-new}).

\begin{theorem}{\cite[pp.~99-106]{gilmer}}\label{may24-05} If $p=0$, 
then ${\rm rad}(I(\mathcal{L}))=I(\mathcal{L})$, and if $p> 0$, then
$${\rm rad}(I(\mathcal{L}))=(t^a-t^b\vert\, p^r(a-b)\in{\mathcal L} 
\mbox{ for some }r\in\mathbb{N}).
$$
\end{theorem}

The degree is independent of the base field $K$.  

\begin{proposition}\label{nov26-12} If $I_\mathbb{Q}(\mathcal{L})$ 
is the lattice ideal of $\mathcal{L}$ over the field $\mathbb{Q}$,
then
$$
\deg(S/I(\mathcal{L}))
=\deg(\mathbb{Q}[t_1,\ldots,t_s]/I_\mathbb{Q}(\mathcal{L})).
$$
\end{proposition}

\begin{proof} 
Let $\succ$ be the GRevLex order on $S$ and
$S_\mathbb{Q}=\mathbb{Q}[t_1,\ldots,t_s]$, and on the extensions
$S[u]$ and $S_\mathbb{Q}[u]$, respectively. Let $G_\mathbb{Q}$ be the
reduced Gr\"obner basis of $I_\mathbb{Q}(\mathcal{L})$. We set
$I=I(\mathcal{L})$ and $I_\mathbb{Q}=I_\mathbb{Q}(\mathcal{L})$. 
Notice that $\mathbb{Z}/p\mathbb{Z}\subset K$, where $p={\rm
char}(K)$.  
Hence $S_{\mathbb{Z}}=\mathbb{Z}[t_1,\ldots,t_s]$ maps into $S$. 
If $G$ denotes the image of $G_\mathbb{Q}$ under this map, 
then using Buchberger's criterion
\cite[p.~84]{CLO}, it is seen that $G$ is a Gr\"obner basis
of $I$. Hence, by Lemma~\ref{elim-ord-hhomog}(a), $G_\mathbb{Q}^h$ and
$G^h$ are Gr\"obner basis of $I_\mathbb{Q}^h$ and $I^h$, respectively,
where $G^h$ is the set of all $f^h$ with $f\in G$. 
Therefore, the rings $S_\mathbb{Q}[u]/I_\mathbb{Q}^h$ and $S[u]/I^h$ 
have the same Hilbert function. Thus, by
Lemma~\ref{elim-ord-hhomog}, the result follows.
\end{proof}

If $I\subset S$ is an ideal and $h\in S$, we set $(I\colon h):=\{f\in
S\vert\, fh\in I\}$. This is the {\it colon ideal} of $I$
relative to $h$. The {\it
saturation\/} of $I$ with respect to $h$ is the ideal $(I\colon
h^{\infty}):=\cup_{k=1}^\infty(I\colon h^k)$.

The following is a well-known result that follows 
from \cite[Corollary~2.5]{EisStu}.  

\begin{theorem}{\rm\cite{EisStu}}\label{jun12-02} Let $I$ be a binomial ideal 
of $S$. Then the following are equivalent.
\begin{center}
{\rm (a)} $I$ is a lattice ideal; {\rm (b)} $I=(I\colon(t_1\cdots
t_s)^\infty)$; {\rm (c)} $t_i$ is a non-zero-divisor of $S/I$ for all $i$.
\end{center}
\end{theorem}

\begin{lemma}{\cite{ci-lattice}}\label{sep1-12} Let 
$\mathcal{L}\subset\mathbb{Z}^s$ be a lattice. Then $\mathcal{L}$ is
generated by $a_1,\ldots,a_m$ if and only if 
$$
((t^{a_1^+}-t^{a_1^-},\ldots,t^{a_m^+}-t^{a_m^-})\colon(t_1\cdots
t_s)^{\infty})=I(\mathcal{L}).
$$
\end{lemma}

If $B$ is a subset of $\mathbb{Z}^s$, $\langle B\rangle$ will denote
the subgroup of  
$\mathbb{Z}^s$ generated by $B$. Let 
$T$ be the 
Laurent polynomial ring $K[\mathbf{t}^{\pm 1}]=K[t_1^{\pm
1},\ldots,t_s^{\pm 1}]$. 
As usual, if $I$ is an ideal of $S$,
$IT$ will denote its extension in
$T$. Part (b) of the next result can be applied to any matrix ideal.

\begin{corollary}\label{nov4-12}
{\rm(a)} If $I(\mathcal{L})=(\{t^{a_i^+}-t^{a_i^-}\}_{i=1}^m)$, then
$\mathcal{L}=\langle\{a_i\}_{i=1}^m\rangle$.

{\rm (b)} If $I=(\{t^{a_i^+}-t^{a_i^-}\}_{i=1}^m)\subset S$ and
$\mathcal{L}=\langle\{a_i\}_{i=1}^m\rangle$, 
then $I(\mathcal{L})=IT\cap S$.
\end{corollary}

\begin{proof} (a): By Theorem~\ref{jun12-02}, 
$(I(\mathcal{L})\colon(t_1\cdots t_s)^\infty)=I(\mathcal{L})$. Hence,
by Lemma~\ref{sep1-12},
$\mathcal{L}=\langle \{a_i\}_{i=1}^m\rangle$. 

(b): See \cite[Corollary~2.5]{EisStu} and its proof. 
\end{proof}

\subsection*{Normal forms and critical groups} Let
$\mathcal{L}\subset\mathbb{Z}^s$ be a lattice of rank $r$ generated by
$a_1,\ldots,a_m$ and let $L$ be the $s\times m$ matrix of rank $r$
with column vectors $a_1,\ldots,a_m$.

There are invertible integer matrices $P$ and $Q$ such that
$PLQ=\Gamma$, where $\Gamma$ is a $s\times m$ ``diagonal'' matrix
$\Gamma={\rm diag}(\gamma_1,\gamma_2,\ldots , \gamma_r,0,\ldots ,0)$,
with $\gamma_i\in\mathbb{N}_+$ and $\gamma_i\mid
\gamma_j$ if $i\leq j$.  The matrix $\Gamma$ is called the {\em normal
  form} of $L$ and the expression $PLQ=\Gamma$ the {\em normal
  decomposition} of $L$ ($\Gamma$ is also called the {\it Smith normal
  form\/} of $L$).  The integers $\gamma_1,\ldots,\gamma_r$ are the
      {\it invariant factors} of $L$. The greatest common divisor of
      all the non-zero $i\times i$ sub-determinants of $L$ will be
      denoted by $\Delta_i(L)$.  The cardinality of a finite set $C$ is
      denoted by $|C|$.

\begin{definition} The group $T(\mathbb{Z}^s/\mathcal{L})$ 
is called the {\it critical group\/} of $\mathcal{L}$ or the {\it
critical group} of $L$.
\end{definition}

Critical groups of lattices will play an important role in several
parts of the paper. Their structure can easily be determined using the
next result, which follows from the {\it fundamental structure theorem
  for finitely generated abelian groups\/} \cite{JacI}. Thus, using
any algebraic system that computes normal forms of matrices, {\it
  Maple\/} \cite{maple} for instance, one can determine the structure
of critical groups.

\begin{theorem}\label{torsion} {\rm (a)} 
\cite[Theorem~3.9]{JacI} $\gamma_1=\Delta_1(L),\
\gamma_i=\Delta_i(L)\Delta_{i-1}(L)^{-1}$ for $i=2,\ldots,r$.

{\rm (b)} \cite[pp.~187-188]{JacI} $\mathbb{Z}^s/\mathcal{L}\simeq
\mathbb{Z}/(\gamma_1)\oplus
\mathbb{Z}/(\gamma_2)\oplus\cdots\oplus
\mathbb{Z}/(\gamma_r)\oplus\mathbb{Z}^{s-r}$.

{\rm (c)} $T(\mathbb{Z}^s/\mathcal{L})\simeq \mathbb{Z}/(\gamma_1)\oplus
\mathbb{Z}/(\gamma_2)\oplus\cdots\oplus\mathbb{Z}/(\gamma_r)$.

{\rm (d)} $\Delta_r(L)=|T(\mathbb{Z}^s/\mathcal{L})|=\gamma_1\cdots \gamma_r$.

{\rm (e)}
$|T(\mathbb{Z}^s/\mathcal{L})|=|T(\mathbb{Z}^m/\mathcal{L}^{\top})|$,
where $\mathcal{L}^{\top}$ is the lattice of $\mathbb{Z}^m$ spanned by the
rows of $L$. 
\end{theorem}

If $m=s$, let $h_{i,j}$ be the
$(i,j)$-minor of $L$, i.e., the determinant of the matrix obtained
from $L$ by eliminating the $i$-th row and the $j$-th column of
$L$. Let $L_{i,j}=(-1)^{i+j}h_{j,i}$ and set 
${\rm adj}(L)=(L_{i,j})$, the {\it adjoint matrix\/} of $L$. Note that
$L_{i,i}=h_{i,i}$. The adjoint matrix will come up later in
Theorem~\ref{GPCB-new} and
Proposition~\ref{gcb-with-strongly-connected-digraph}. 

\section{Degree and torsion in primary
decompositions}\label{prim-dec-deg-tor}

In this section, we study primary decompositions of lattice ideals
over an arbitrary field, using the Eisenbud-Sturmfels theory of
binomial ideals over algebraically closed fields \cite{EisStu}. For
a graded lattice ideal of dimension $1$, we give the explicit minimal
primary decomposition over a field with enough roots of unity. 

\begin{lemma}\label{sep7-12} Let $F$ be a field extension of 
$K$, let $B=F[t_1,\ldots,t_s]$ be a polynomial ring with coefficients 
in $F$ and let $I$ be an ideal of $S$. Then the following hold. 
\begin{itemize}
\item[\rm(a)] $IB\cap S=I$.
\item[\rm(b)] $\deg(S/I)=\deg(B/IB)$.
\item[\rm(c)] If $\mathfrak{q}$ is a $\mathfrak{p}$-primary ideal of
  $B$, then $\mathfrak{q}\cap S$ is a $\mathfrak{p}\cap S$-primary
  ideal of $S$.
\item[\rm(d)] If $IB=\cap_{i=1}^{r}\mathfrak{q}_i$ is a primary
  decomposition of $IB$, then $I=\cap_{i=1}^r(\mathfrak{q}_i\cap S)$
  is a primary decomposition of $I$ such that ${\rm
    rad}(\mathfrak{q}_i\cap S)={\rm rad}(\mathfrak{q}_i)\cap S$.
\item[\rm(e)] If $I$ is the lattice ideal of $\mathcal{L}$ in $S$, then 
$IB$ is the lattice ideal of $\mathcal{L}$ in $B$. 
\end{itemize}
\end{lemma}

\begin{proof} Let $G=\{g_1,\ldots,g_m\}$ be a Gr\"obner basis of $I$
with respect to the GRevLex order. By Buchberger's criterion
\cite[Theorem~6, p.~84]{CLO}, 
$G$ is also a Gr\"obner basis for $IB$ with respect to the GRevLex
order on $B$. 

(a): Note that $K\hookrightarrow F$ is a faithfully flat extension.
Apply the functor $(-)\otimes _{K}S.$ By Base Change, it follows that
$S\hookrightarrow B$ is a faithfully flat extension. Hence, by
\cite[Theorem~7.5(ii)]{Mats}, $I=IB\cap S$ for any ideal $I$ of $S$.

(b): By Lemma~\ref{elim-ord-hhomog}(a), $I^h$ and $IB^h$ are both
generated by $G^h=\{g_1^h,\ldots,g_m^h\}$, and $G^h$ is a Gr\"obner
basis for $IB^h$. Hence, by Lemma~\ref{elim-ord-hhomog}(b), we get
$$
H_I^a(i)=H_{I^h}(i)=\dim_K(S[u]/I^h)_i=
\dim_F(B[u]/IB^h)_i=H_{IB^h}(i)=H_{IB}^a(i)
$$
for $i\geq 0$. Therefore, one has the equality 
$\deg(S/I)=\deg(B/IB)$.

(c): This is well-known and not hard to show.  

(d): This follows from  (a) and (c). 

(e): Let $t^{a_1^+}-t^{a_1^-},\ldots,t^{a_m^+}-t^{a_m^-}$ be a set of
generators of $I$. This set also generates $IB$.  
Then, by Lemma~\ref{sep1-12}, one has 
\begin{equation}\label{sep7-12-1}
(IB\colon_B\,\, (t_1\cdots
t_s)^\infty)=I_B,\tag{$*$}
\end{equation}
where $I_B$ is the lattice ideal of $\mathcal{L}$ in $B$. We claim
that $t_i$ is not a zero-divisor of $B/IB$ for $i=1,\ldots,s$.  Note
that since $S\hookrightarrow B$ is a flat extension, applying the
functor $(-)\otimes _{S}S/I$ and using Base Change, we deduce that
$S/I\hookrightarrow B/IB$ is a flat extension. Hence since the map
$S/I\stackrel{t_i}{\longrightarrow} S/I$ is injective, by
Theorem~\ref{jun12-02}, so is the map
$B/IB\stackrel{t_i}{\longrightarrow} B/IB$. Consequently, by the
claim, the left hand side of Eq.~(\ref{sep7-12-1}) is equal to $IB$
and we get the equality $IB=I_B$.
\end{proof}

We come to the first main result of this section. 

\begin{theorem}\label{bounds-for-the-number-of-associated-primes} 
Let $I(\mathcal{L})$ be a lattice ideal of $S$ over an arbitrary field
$K$ of characteristic $p$, let $c$ be the number of associated primes
of $I(\mathcal{L})$, and for $p>0$, let $G$ be the unique largest
subgroup of $T(\mathbb{Z}^s/\mathcal{L})$ whose order is relatively
prime to $p$. Then
\begin{itemize}
\item[{\rm (a)}] All associated primes of $I(\mathcal{L})$ have height
equal to ${\rm rank}(\mathcal{L})$.

\item[{\rm (b)}] $|T(\mathbb{Z}^s/\mathcal{L})|\geq c$ if $p=0$ and
$|G|\geq c$ if $p>0$, with equality if $K$ 
is algebraically closed.

\item[{\rm (c)}] $\deg(S/I(\mathcal{L}))\geq |T(\mathbb{Z}^s/\mathcal{L})|$
if $p=0$ and $\deg(S/I(\mathcal{L}))\geq |G|$ if 
$p>0$.
\end{itemize}
\end{theorem}

\begin{proof} Let $\overline{K}$ be the algebraic closure of $K$ and
let $\overline{S}=\overline{K}[t_1,\ldots,t_s]$ be the corresponding
polynomial ring with coefficients in $\overline{K}$. Thus, we have an
integral extension $S\subset \overline{S}$ of normal domains. We set
$I=I(\mathcal{L})$ and $\overline{I}=I\overline{S}$, where the latter
is the extension of $I$ to $\overline{S}$. The ideal $\overline{I}$ is
the lattice ideal of $\mathcal{L}$ in $\overline{S}$ (see
Lemma~\ref{sep7-12}(e)). Hence, as $\overline{K}$ is algebraically
closed, by \cite[Corollaries~2.2 and 2.5]{EisStu} $\overline{I}$ has a
unique minimal primary decomposition
\begin{equation}\label{sep9-12}
\overline{I}=\overline{\mathfrak{q}}_1
\cap\cdots\cap\overline{\mathfrak{q}}_{c_1},\tag{$\dag$}
\end{equation}
where $c_1=|T(\mathbb{Z}^s/\mathcal{L})|$ if $p=0$ and $c_1=|G|$ if 
$p>0$. Notice that $c_1=|T(\mathbb{Z}^s/\mathcal{L})|$ if
$p$ is relatively prime to $|T(\mathbb{Z}^s/\mathcal{L})|$.
Furthermore, also by \cite[Corollaries~2.2 and 2.5]{EisStu}, one has
that if $\overline{\mathfrak{p}}_i={\rm rad}(\overline{\mathfrak{q}}_i)$
for $i=1,\ldots,c_1$, then
$\overline{\mathfrak{p}}_1,\ldots,\overline{\mathfrak{p}}_{c_1}$ are 
the associated primes of $\overline{I}$ and ${\rm
ht}(\overline{\mathfrak{p}}_i)={\rm rank}(\mathcal{L})$ for
$i=1,\ldots,c_1$. Hence, by Lemma~\ref{sep7-12}(d), one has a primary
decomposition
\begin{equation}\label{dec7-13}
I=(\overline{\mathfrak{q}}_1\cap S)
\cap\cdots\cap(\overline{\mathfrak{q}}_{c_1}\cap S)\tag{\ddag}
\end{equation}
such that ${\rm rad}(\overline{\mathfrak{q}}_i\cap
S)={\rm rad}(\overline{\mathfrak{q}}_i)\cap
S=\overline{\mathfrak{p}}_i\cap S$. We set
$\mathfrak{p}_i=\overline{\mathfrak{p}}_i\cap S$ and  
$\mathfrak{q}_i=\overline{\mathfrak{q}}_i\cap S$
for $i=1,\ldots,c_1$. 

(a): Since $S$ is a normal domain and $S\subset
\overline{S}$ is an integral extension, we get ${\rm
ht}(\mathfrak{p}_i)={\rm ht}(\overline{\mathfrak{p}}_i)={\rm
rank}(\mathcal{L})$ for all $i$
(see \cite[Theorems~5 and 20]{Mat}). 

(b): By Eq.~(\ref{dec7-13}), the associated primes of $I$
are contained in $\{\mathfrak{p}_1,\ldots,\mathfrak{p}_{c_1}\}$. 
Thus, $c_1\geq c$, which proves the
first part. Now, assume that $K=\overline{K}$. By (a), 
we may assume that $\mathfrak{p}_1,\ldots,\mathfrak{p}_c$ 
are the minimal primes of $I$. Consequently $I$ has a unique minimal primary
decomposition $I=\mathfrak{Q}_1\cap\cdots\cap\mathfrak{Q}_c$ 
such that $\mathfrak{Q}_i$ is $\mathfrak{p}_i$-primary and ${\rm
ht}(\mathfrak{p}_i)={\rm rank}(\mathcal{L})$ for $i=1,\ldots,c$. As
$I=\overline{I}$, from Eq.~(\ref{sep9-12}), we get that $c_1=c$.

(c): Using Lemma~\ref{sep7-12}(b), Eq.~(\ref{sep9-12}) that
was stated at the beginning of the proof, and
the additivity of the degree
(Proposition~\ref{additivity-of-the-degree}), we get
$$
\textstyle
\deg(S/I)=\deg(\overline{S}/\overline{I})=
\sum_{i=1}^{c_1}\deg(\overline{S}/\overline{\mathfrak{q}}_i)\geq c_1,
$$
as required.
\end{proof}

\subsection*{Primary decompositions of graded lattice ideals} 
The aim here is to give explicitly the minimal primary decomposition
of a graded lattice 
ideal of dimension $1$ in terms of the normal decomposition of an 
integer matrix (see Section~\ref{prelim}).

Let $\mathcal{L}$ be a lattice in $\mathbb{Z}^s$ of rank $s-1$
generated by $a_1,\ldots,a_m$ and let $L$ be the $s\times m$ matrix
with column vectors $a_1,\ldots,a_m$.  There are invertible integer
matrices 
$P=(p_{ij})$ and $Q$ such that $PLQ=\Gamma$, where 
$\Gamma={\rm diag}(\gamma_1,\gamma_2,\ldots
, \gamma_{s-1},0,\ldots ,0)$, with $\gamma_i\in\mathbb{N}_+$ and
$\gamma_i\mid\gamma_j$ 
if $i\leq j$. The torsion subgroup of $\mathbb{Z}^s/\mathcal{L}$ has
order $\gamma:=\gamma_1\cdots\gamma_{s-1}$. 
If $p_{i,*}$ denotes the $i$-th row of $P$, then the last row of $P$
satisfies $\gcd(p_{s,*})=1$,
$p_{s,*}L=0$ and $\mathcal{L}\subset{\rm ker}(p_{s,*})$. This follows
using the technique described in \cite[p.~37]{New} to solve systems
of linear equations over the integers.  Thus ${\rm
ker}(p_{s,*})$ is equal to $\mathcal{L}_s$, the saturation of $\mathcal{L}$ in
the sense of \cite{EisStu}. For convenience recall that 
$\mathcal{L}_s$ is the set of all $a\in\mathbb{Z}^s$ such that
${\eta}a\in\mathcal{L}$ for some $0\neq\eta\in\mathbb{Z}$. 

For the rest of this section we suppose that $K$ is a field 
containing the $\gamma_{s-1}$-th roots of unity with $\ch(K)=0$ or
$\ch(K)=p$, $p$ a prime with  
$p\nmid \gamma_{s-1}$. Under this assumption, the lattice ideal
$I(\mathcal{L})$ is radical because $\ch(K)=0$ or 
$\gcd(p,|T(\mathbb{Z}^s/\mathcal{L}|)=1$ (see
Theorem~\ref{may24-05}), 
and for each
$i$ the polynomial $z^{\gamma_i}-1$ has $\gamma_i$ distinct roots in
$K$. Write $\Lambda_i$ to denote the set of $\gamma_i$-th roots
of unity in $K$ and $\Lambda=\prod_{i=1}^{s-1}\Lambda_i$. The set
$\Lambda$ is a group under componentwise multiplication and
$|\Lambda|=\gamma$. We also
suppose that $I(\mathcal{L})$ is graded with respect to a weight
vector $\mathbf{d}=(d_1,\ldots,d_s)$ in $\mathbb{N}_+^s$ with
$\gcd(\mathbf{d})=1$. 
Notice that $\mathbf{d}=\pm p_{s,*}$. This means that in the
non-graded case, $\pm p_{s,*}$ can play the role of $\mathbf{d}$.
Consider a polynomial ring $K[x_1]$ in one variable. For any
$\lambda=(\lambda_1,\ldots,\lambda_{s-1})\in\Lambda$, 
there is an homomorphism of $K$-algebras
$$
\varphi_{\lambda}\colon S\to K[x_1],\ \ \ t_i\longmapsto\lambda_1^{p_{1,i}}\cdots
  \lambda_{s-1}^{p_{s-1,i}}x_1^{d_i}, \ \ \ i=1,\ldots ,s.
$$
The kernel of $\varphi_\lambda$, denoted by  $\mathfrak{a}_{\lambda}$,
is a prime ideal of $S$ of height $s-1$ \cite{opPCB}. 
This type of ideal was  
introduced in \cite{opPCB} to study the algebraic properties of PCB ideals. 

The primary decomposition of an arbitrary lattice ideal over an
algebraically closed field is given in \cite[Corollary~2.5]{EisStu}.
This decomposition will be used to prove
Theorem~\ref{degree-lattice-char=any}(c).    
The next result shows an explicit primary decomposition for the
class of ideals under consideration. 

\begin{theorem}\label{main-new}  If $I(\mathcal{L})$ is a graded
lattice ideal of dimension $1$, then 
$I(\mathcal{L})=\cap_{\lambda\in\Lambda}\mathfrak{a}_{\lambda}$ is 
the minimal primary decomposition of $I(\mathcal{L})$ 
into exactly $\gamma$ primary components. 
\end{theorem}

\begin{proof} First of all, recall from Theorem~\ref{may24-05} that
$I(L)$ is a radical 
ideal. For any
$\lambda=(\lambda_1,\ldots,\lambda_{s-1})$ in $\Lambda$, let
$\rho_\lambda\colon \mathcal{L}_s\rightarrow K^*$ be the partial character
of $\mathcal{L}_s$ given by
$$
\rho_\lambda(\alpha)=
\textstyle \prod_{i=1}^{s-1}\lambda_i^{p_{i,1}\alpha_1+\cdots+p_{i,s}\alpha_s},\
\ \ \mbox{ where }\ \alpha=(\alpha_1,\ldots,\alpha_s). 
$$
Clearly
$\rho_\lambda$ is a group homomorphism.  First we show the
following two conditions: 
(i) $\rho_\lambda(\alpha)=1$ for $\alpha\in\mathcal{L}$; 
(ii) If $\rho_\lambda(\alpha)=1$ for all $\alpha\in\mathcal{L}_s$, 
then $\lambda=(1,\ldots,1)$. 

(i): Take
$\alpha=(\alpha_1,\ldots,\alpha_s)$ in $\mathcal{L}$. To show 
$\rho_\lambda(\alpha)=1$ it suffices to prove:
\begin{equation}\label{feb20-13-1} 
{p_{i,1}\alpha_1+\cdots+p_{i,s}\alpha_s}\,\equiv\, 0\ {\rm mod}(\gamma_i)\ 
\mbox{ for all }i.
\end{equation}
As $\alpha$ is a linear combination of the columns of $L$, 
one has $\alpha^\top=L\mu^\top=P^{-1}\Gamma Q^{-1}\mu^\top$ 
for some $\mu\in\mathbb{Z}^m$.
Hence
$$
P\alpha^\top=\Gamma
Q^{-1}\mu^\top=(\eta_1\gamma_1,\ldots,\eta_{s-1}\gamma_{s-1},0)^\top 
$$
for some $\eta_i$'s in $\mathbb{Z}$. Thus Eq.~(\ref{feb20-13-1})
holds, as required. 

(ii): Fix $k$ such that $1\leq k\leq s-1$. We set
$P^{-1}=(r_{ij})$ and denote by $r_{*,k}$ the $k$-th column of
$P^{-1}$. Then $P^{-1}e_k^\top=r_{*,k}$ and $Pr_{*,k}=e_k^\top$, 
where $e_k$ is the $k$-th unit vector. Since $PP^{-1}=I$, 
we get that $\langle p_{i,*},r_{*,k}\rangle$ is $1$ if $i=k$ and is
$0$ otherwise. In particular $\langle p_{s,*},r_{*,k}\rangle=0$, 
i.e, $r_{*,k}$ is in $\mathcal{L}_s={\rm ker}(p_{s,*})$. Setting
$\alpha^\top=r_{*,k}$ 
and using that $\rho_\lambda(\alpha)=1$, we get
that $\lambda_k=1$.

Let $I_\lambda$ be the ideal of $S$ generated by all
$t^{\alpha^+}-\rho_\lambda(\alpha)t^{\alpha^-}$ with
$\alpha\in\mathcal{L}_s$. Let us see that $I(\mathcal{L})\subset
I_\lambda\subset\mathfrak{a}_\lambda$ for all $\lambda\in\Lambda$. The
first inclusion follows from (i). To show the second inclusion take
$f=t^{\alpha^+}-\rho_\lambda(\alpha)t^{\alpha^-}$ in $I_\lambda$, with
$\alpha\in\mathcal{L}_s$. It is easy to see that $f\in{\rm
  ker}(\varphi_\lambda)$ using that $\alpha\in{\rm ker}(\mathbf{d})$
and the definition of $\rho_\lambda$. Hence, as $\mathfrak{a}_\lambda$
is a prime ideal of height $s-1$ for all $\lambda\in \Lambda$, we get
that $\mathfrak{a}_\lambda$ is a primary component of $I(\mathcal{L})$
for all $\lambda\in\Lambda$.  We claim that
$\mathfrak{a}_\lambda\neq\mathfrak{a}_\upsilon$ if
$\lambda\neq\upsilon$ and $\lambda,\upsilon\in\Lambda$. To show this
assume that $\mathfrak{a}_\lambda=\mathfrak{a}_\upsilon$. Take an
arbitrary $\alpha=(\alpha_i)$ in $\mathcal{L}_s$. Then
$f=t^{\alpha^+}-\rho_\lambda(\alpha)t^{\alpha^-}$ is in
$I_\lambda\subset \mathfrak{a}_\lambda$. Thus
$f\in\mathfrak{a}_\upsilon={\rm ker}(\varphi_\upsilon)$. It follows
readily that $\rho_\lambda(\alpha)=\rho_\upsilon(\alpha)$. Therefore
$\rho_{\lambda\upsilon^{-1}}(\alpha)=1$ for all
$\alpha\in\mathcal{L}_s$, and by (ii) $\lambda=\upsilon$. This proves
the claim. Altogether $I(\mathcal{L})$ has at least $\gamma$ prime
components and by
Theorem~\ref{bounds-for-the-number-of-associated-primes}(b) it has at
most $\gamma$ prime components. Thus
$I(\mathcal{L})=\cap_{\lambda\in\Lambda}\mathfrak{a}_{\lambda}$.
\end{proof}

\begin{remark} The ideal $I_\lambda$ is called the lattice ideal
of $\mathcal{L}_s$ relative to the partial character $\rho_\lambda$
\cite{EisStu}.  Since $\mathfrak{a}_\lambda$ is a prime ideal
generated by ``binomials'' of the form $t^{\alpha^+}-\eta
t^{\alpha^-}$, with $\eta\in K^*$, it follows that the inclusion
$I_\lambda\subset\mathfrak{a}_\lambda$ is an equality.
\end{remark}

\section{The degree of lattice and toric ideals}\label{degree-section}

In this section we show that the homogenization of a lattice ideal (resp.
toric ideal) is again a lattice ideal (resp. toric ideal). For any
toric or lattice ideal, we give formulae to
compute the degree in terms of the torsion of certain factor groups of
$\mathbb{Z}^s$ and in terms of relative volumes of lattice polytopes.
A general reference for connections between monomial subrings, Ehrhart
rings, polyhedra and volume is \cite[Chapters~5 and
6]{BHer} (see also \cite{ehrhart,Stur1,handbook}).

Let $\mathcal{L}\subset\mathbb{Z}^s$ be a lattice and let $e_{i}$ be the
$i$-th unit vector in $\mathbb{Z}^{s+1}$. For
$a=(a_i)\in\mathbb{Z}^s$ define the {\it value\/} of $a$ as
$|a|=\sum_{i=1}^sa_i$, and the {\it homogenization\/} of $a$ with
respect to $e_{s+1}$ as 
$a^h=(a,0)-|a|e_{s+1}$ if $|a|\geq 0$ and
$a^h=(-a)^h$ if $|a|<0$. The choice of the last coordinate of $a^h$
is a convenience. The {\it homogenization\/} of $\mathcal{L}$,
denoted by $\mathcal{L}^h$, is the lattice of $\mathbb{Z}^{s+1}$
generated by all $a^h$ such that $a\in\mathcal{L}$.  

\begin{lemma}\label{o'carroll-comments} 
Let $I(\mathcal{L})\subset S$ be a lattice ideal and let
$f_1,\ldots,f_r$ be a set
of binomials such that the terms of $f_i$ have disjoint support for 
all $i$. Then the following hold.
\begin{itemize}
\item[(a)] $I(\mathcal{L})^h\subset S[u]$ is a
lattice ideal. 
\item[(b)] If $\mathcal{L}=\langle b_1,\ldots,b_r\rangle$, then
$\mathcal{L}^h=\langle b_1^h,\ldots,b_r^h\rangle$.
\item[(c)] $I(\mathcal{L}^h)=I(\mathcal{L})^h$.
\item[(d)] If $((f_1,\ldots,f_r)\colon(t_1\cdots
t_s)^\infty)=I(\mathcal{L})$, then $((f_1^h,\ldots,f_r^h)\colon(t_1\cdots
t_su)^\infty)=I(\mathcal{L})^h$.
\end{itemize}
\end{lemma}

\begin{proof} We set $I=I(\mathcal{L})$. Let $\succ$ denote the GRevLex
order on $S$ and on $S[u]$, and let
$g_1,\ldots,g_m$ be the reduced Gr\"obner basis of $I$. By
Lemma~\ref{elim-ord-hhomog}(a), $g_1^h,\ldots,g_m^h$ is a Gr\"obner
basis of $I^h$.

(a): As $I^h$ is a binomial ideal, by Theorem~\ref{jun12-02} we need
only show that $t_i$ is a non-zero-divisor of $S[u]/I^h$ for
$i=1,\ldots,s+1$, where $t_{s+1}=u$. First, we show the case
$i=s+1$. Assume that $uf\in I^h$ for some $f\in S[u]$. By the Division
Algorithm \cite[Theorem~2.11]{Ene-Herzog}, we can write
$$
f=h_1g_1^h+\cdots+h_mg_m^h+\mathfrak{f},
$$ where $h_1,\ldots,h_m,\mathfrak{f}$ are in $S[u]$ and
$\mathfrak{f}$ is not divisible by any of the leading terms of
$g_1^h,\ldots,g_m^h$.  We claim that $\mathfrak{f}=0$. If
$\mathfrak{f}\neq 0$, then $u\mathfrak{f}\in I^h$ and consequently
$u\, {\rm in}(\mathfrak{f})\in{\rm in}(I^h)$, where ${\rm
  in}(\mathfrak{f})$ is the leading term of $\mathfrak{f}$ and ${\rm
  in}(I^h)$ is the initial ideal of $I^h$.  Hence, $u\, {\rm
  in}(\mathfrak{f})$ is a multiple of ${\rm in}(g_\ell^h)$ for some
$\ell$. Since $u$ does not appear in the monomial ${\rm
  in}(g_\ell)={\rm in}(g_\ell^h)$, we get that ${\rm
  in}(\mathfrak{f})$ is divisible by ${\rm in}(g_\ell^h)$, a
contradiction.  Thus, $\mathfrak{f}=0$ and $f\in I^h$, as
required. Next, we show the case $1\leq i\leq s$. Assume that
$t_{i}g\in I^h$ for some $i$ and some $g$ in $S[u]$. As $I^h$ is
graded, we may assume that $g$ is homogeneous of degree
$\delta$. Since $u$ is a non-zero-divisor of $S[u]/I^h$, we may assume
that $u$ does not divides the leading term of $g$. If
$f=g(t_1,\ldots,t_s,1)$, then $\deg(f)=\delta$, $f^h=g$ and $t_{i}f\in
I$. Hence $f\in I$. Therefore, as $g_1,\ldots,g_m$ is a Gr\"obner
basis, we can write $f=h_1g_1+\cdots+h_mg_m$, where
$\deg(f)\geq\deg(h_jg_j)$ for all $j$. It is not hard to see that
$g=f^h\in I^h$, as required.

(b): By changing the sign of $b_i$, if necessary, we may assume that
$|b_i|\geq 0$ for all $i$. Clearly $\mathcal{L}^h$ contains $\langle
b_1^h,\ldots,b_r^h\rangle$. Indeed, since $b_i\in\mathcal{L}$ for all
$i$, we get $b_i^h\in\mathcal{L}^h$ for all $i$. Now, we prove
that $\mathcal{L}^h$ is contained in $\langle
b_1^h,\ldots,b_r^h\rangle$.  
It suffices to show that $c^h\in\langle b_1^h,\ldots,b_r^h\rangle$ for
any $c\in\mathcal{L}$ with $|c|\geq 0$. By hypothesis, we can write
$c=\lambda_1b_1+\cdots+\lambda_rb_r$ for some integers
$\lambda_1,\ldots,\lambda_r$. Hence
$|c|=\lambda_1|b_1|+\cdots+\lambda_r|b_r|$. From the last two
equalities, we obtain that $c^h$ is equal to
$\lambda_1b_1^h+\cdots+\lambda_rb_r^h$, as required.

(c): We can write $g_i=t^{a_i^+}-t^{a_i^-}$, with $|a_i|\geq 0$, for 
$i=1,\ldots,m$. By Corollary~\ref{nov4-12}, $\mathcal{L}$ is generated
by $a_1,\ldots,a_m$. Then, by part (b),  $\mathcal{L}^h$ is
generated by $a_1^h,\ldots,a_m^h$. Notice that $g_i^h$ is equal to 
$t^{(a_i^h)^+}-t^{(a_i^h)^-}$ for all $i$. Therefore, using part (a)
and Lemma~\ref{sep1-12}, we get
\begin{eqnarray*}
I(\mathcal{L})^h&=&(I(\mathcal{L})^h\colon(t_1\cdots
t_su)^\infty)\\
&=&((t^{(a_1^h)^+}-t^{(a_1^h)^-},\ldots,t^{(a_m^h)^+}-t^{(a_m^h)^-})
\colon(t_1\cdots t_su)^\infty)=I(\mathcal{L}^h).
\end{eqnarray*}

(d): This part follows from Lemma~\ref{sep1-12} and part (b).
\end{proof}

Let $H=\{x^{v_1},\ldots,x^{v_s}\}\subset 
K[x_1^{\pm 1},\ldots,x_n^{\pm 1}]$ 
be a set of Laurent monomials, where $v_i\in\mathbb{Z}^n$, and let
$K[H]$ be the $K$-subalgebra 
generated by $H$. There is an epimorphism of $K$-algebras
\[ 
\varphi\colon S=K[t_1,\ldots,t_s] \longrightarrow K[H], \ \ \ 
t_i\longmapsto x^{v_i}.
 \]
The kernel of $\varphi$, denoted by $P$, is called the {\it toric
ideal\/} of  
$K[H]$. In general, $P$ is not a graded ideal. 
Since $S/P\simeq K[H]$, the {\it degree} of
$K[H]$ is defined to be the degree of $S/P$. 

\begin{lemma}\label{oct4-12} 
Let $P^h\subset S[u]$ be the
homogenization of $P$. Then the following hold.
\begin{itemize}
\item[{\rm(a)}] $P^h$ is the toric 
ideal of $K[H']:=K[z,x^{v_2-v_1}z,\ldots,x^{v_s-v_1}z,x^{-v_1}z]$.

\item[{\rm(b)}] The toric ideal of $K[x^{v_1}z,\ldots,x^{v_s}z,z]$ is
the toric ideal $P^h$ of $K[H']$. 
\end{itemize}
\end{lemma}

\begin{proof} (a): The toric ideal of $K[H']$, denoted by $P'$, is
the kernel of the map 
$\varphi\colon S[u]\rightarrow K[H']$ induced by $t_i\mapsto
x^{v_i-v_1}z$ for 
$i=1,\ldots,s+1$, where $v_{s+1}=0$ and $t_{s+1}=u$. Let $G$ be the
reduced Gr\"obner basis of $P$ with respect to the GRevLex order.
First, we show the
inclusion $P^h\subset P'$. Take an element $f$ of $G$. By
Lemma~\ref{elim-ord-hhomog}(a), it suffices to 
show that $f^h$ is in $P'$. We can write
$f=t^{a^+}-t^{a^-}$ with ${\rm in}(f)=t^{a^+}$. Thus,
$|a^+|\geq|a^{-}|$ and
$f^h=t^{a^+}-t^{a^-}u^{|a|}$, where $a=a^+-a^{-}$. We set
$a=(a_1,\ldots,a_s)$. 
From the equality
$$
0=a_1v_1+\cdots+a_sv_s=a_2(v_2-v_1)+\cdots+
a_s(v_s-v_1)+(a_1+\cdots+a_s)v_1,
$$
we get that $f^h\in P'$, as required. Let $G'$ be the
reduced Gr\"obner basis of $P'$ with respect to the GRevLex order.
Next we show the
inclusion $P'\subset P^h$. Take an element $f'$ of $G'$. It suffices to 
show that $f'$ is in $P^h$. As $f'$ is homogeneous in the standard
grading of $S[u]$, we can write $f'=t^{c^+}-t^{c^-}u^{|c|}$ with
${\rm in}(f')=t^{c^+}$ and $c=(c_1,\ldots,c_s)$.  Since $f'\in P'$, we get
$$
(z^{c_1^+})(x^{v_2-v_1}z)^{c_2^+}\cdots (x^{v_s-v_1}z)^{c_s^+}
=(z^{c_1^-})(x^{v_2-v_1}z)^{c_2^-}\cdots
(x^{v_s-v_1}z)^{c_s^-}(x^{-v_1}z)^{|c|}.
$$
Hence, $c_2(v_2-v_1)+\cdots+c_s(v_s-v_1)=-|c|v_1$. Consequently,
$c_1v_1+\cdots+c_sv_s=0$, that is, the binomial $f=t^{c^+}-t^{c^-}$ is
in $P$. As $f'=f^h$, we get $f'\in P^h$. 

(b): The map that sends $z$ to $x^{v_1}z$ induces an
isomorphism $K[H']\to K[H'']$. So this part is an immediate 
consequence of (a).
\end{proof}

\begin{proposition}{\rm(\cite[Proposition~3.5]{ehrhart},
\cite[Corollary~5.35]{handbook})}\label{apr3-01} Let
  $\mathcal{A}=\{\alpha_1,\ldots,\alpha_m\}$ be a set of points of
  $\mathbb{Z}^n$ and let $\mathcal{P}={\rm conv}({\mathcal A})$ be the
  convex hull of $\mathcal{A}$. Then
$$
|T(\mathbb{Z}^{n}/(\alpha_1-\alpha_m,\ldots,\alpha_{m-1}-\alpha_m))| 
\deg(K[x^{\alpha_1}z,\ldots,x^{\alpha_m}z])=r!{\rm
vol}(\mathcal{P}),
$$
where $r=\dim(\mathcal{P})$, ${\rm vol}(\mathcal{P})$ is the
relative volume of $\mathcal{P}$ and $z$ is a new variable.
\end{proposition}

\begin{definition}
The term $r!{\rm vol}(\mathcal{P})$ is called the {\it normalized 
volume\/} of $\mathcal{P}$.
\end{definition}

The next
result holds for any toric ideal. 

\begin{theorem}\label{toric-degree-non-homg} Let $P$ be the toric ideal of
$K[H]=K[x^{v_1},\ldots,x^{v_s}]$, let $A$ be the $n\times s$ matrix
with column vectors $v_1,\ldots,v_s$ and let $r$ be the rank of
 $A$. Then 
$$
|T(\mathbb{Z}^n/\langle v_1,\ldots,v_s\rangle)|\deg(S/P)=r!{\rm
vol}({\rm conv}(v_1,\ldots,v_s,0)). 
$$
\end{theorem}

\begin{proof} By Lemma~\ref{elim-ord-hhomog}(c), 
$\deg(S/P)=\deg(S[u]/P^h)$. On the other hand, by
Lemma~\ref{oct4-12}(b),
$P^h$ is the toric ideal of the monomial subring: 
$$
K[H'']=K[x^{v_1}z,\ldots,x^{v_s}z,z].
$$
Hence $S[u]/P^h\simeq K[H'']$. 
Therefore, setting $m=s+1$, $\alpha_i=v_i$ for $i=1,\ldots,m-1$,
$\alpha_m=0$ and $\mathcal{A}=\{\alpha_1,\ldots,\alpha_m\}$, 
the result follows readily from Proposition~\ref{apr3-01}. 
\end{proof}

We come to the main result of this section.

\begin{theorem}\label{degree-lattice-char=any} Let 
$\mathcal{L}\subset\mathbb{Z}^s$ be a lattice of 
rank $r$. Then the following hold. 
\begin{itemize}

\item[{\rm (a)}] If $r=s$, then $\deg(S/I(\mathcal{L}))
=|\mathbb{Z}^s/\mathcal{L}|$.  
 
\item[ {\rm (b)}] If $r<s$, there is an integer matrix $A$ of size
  $(s-r)\times s$ with ${\rm rank}(A)=s-r$ such that we have the
  containment of rank $r$ lattices ${\mathcal L}\subset {\rm
    ker}_\mathbb{Z}(A)$, with equality if and only if
  $\mathbb{Z}^s/{\mathcal L}$ is torsion-free.

\item[{\rm (c)}] If $r<s$ and $v_1,\ldots,v_s$ are the columns of $A$, then
\begin{eqnarray*}
\deg(S/I(\mathcal{L}))&=&
\frac{|T(\mathbb{Z}^s/\mathcal{L})|(s-r)!{\rm
vol}({\rm conv}(0,v_1,\ldots,v_s))}{|T(\mathbb{Z}^{s-r}/\langle 
v_1,\ldots,v_s\rangle)|}.
\end{eqnarray*}
\end{itemize}
\end{theorem}

\begin{proof} (a): By Lemma~\ref{elim-ord-hhomog}(c),
$\deg(S/I(\mathcal{L}))=\deg(S[u]/I(\mathcal{L})^h)$, and by
Lemma~\ref{o'carroll-comments}(c),
$I(\mathcal{L})^h=I(\mathcal{L}^h)$. Since $S[u]/I(\mathcal{L}^h)$ has
dimension $1$, by \cite[Theorem~3.12]{degree-lattice} the degree of 
$S[u]/I(\mathcal{L}^h)$ is $|T(\mathbb{Z}^{s+1}/\mathcal{L}^h)|$. 
Let $t^{a_1^+}-t^{a_1^-},\ldots,t^{a_m^+}-t^{a_m^-}$ be a set of
generators of $I(\mathcal{L})$. By Corollary~\ref{nov4-12},
$\mathcal{L}$ is generated 
by $a_1,\ldots,a_m$. We may assume that $|a_i|\geq 0$ for all $i$. 
Then, by Lemma~\ref{o'carroll-comments}(b),
$\mathcal{L}^h$ is generated by $a_1^h,\ldots,a_m^h$. Let $A$ and
$A^h$ be the matrices with rows $a_1,\ldots,a_m$ and
$a_1^h,\ldots,a_m^h$, respectively. Notice that $A^h$ is obtained 
from $A$ by adding the column vector $b=(-|a_1|,\ldots,-|a_m|)^\top$. 
Since $b$ is a linear combination of the columns of $A$, by the
fundamental theorem of finitely generated abelian groups (see
\cite[p.~187]{JacI} and Theorem~\ref{torsion}), we get 
that the groups $\mathbb{Z}^s/\mathcal{L}$ and
$\mathbb{Z}^{s+1}/\mathcal{L}^h$ 
have the same torsion. Thus, $|\mathbb{Z}^s/\mathcal{L}|$ is equal to
$|T(\mathbb{Z}^{s+1}/\mathcal{L}^h)|$. Altogether, the degree of
$S/I(\mathcal{L})$ is the order of $\mathbb{Z}^s/\mathcal{L}$.

(b): We may assume that $\mathcal{L}=
\mathbb{Z}\alpha_1\oplus\cdots\oplus\mathbb{Z}\alpha_r$, where 
$\alpha_1,\ldots,\alpha_s$ is a $\mathbb{Q}$-basis of 
$\mathbb{Q}^s$. Consider the hyperplane 
$H_i$ of $\mathbb{Q}^s$ generated by 
$\alpha_1,\ldots,\widehat{\alpha}_i,\ldots,\alpha_s$. Note 
that the subspace of $\mathbb{Q}^s$ generated by 
$\alpha_1,\ldots,\alpha_r$ is equal to 
$H_{r+1}\cap\cdots\cap H_s$. There is a normal vector 
$w_i\in\mathbb{Z}^s$ 
such that 
$$
H_i=\{\alpha\in\mathbb{Q}^s\vert\, \langle\alpha,w_i\rangle=0\}.
$$
It is not hard to see that the matrix $A$ with rows 
$w_{r+1},\ldots,w_s$ is the matrix with the required 
conditions, because,  
by construction, $\alpha_i\in H_j$ for $i\neq j$ and 
consequently $w_{r+1},\ldots,w_s$ are linearly independent. In
particular we have the equality ${\rm rank}(\mathcal{L})={\rm rank}
({\rm ker}_\mathbb{Z}(A))$. 

(c): By Proposition~\ref{nov26-12} and Lemma~\ref{sep7-12}(b), we may
assume that $K$ is algebraically closed of characteristic zero. Let
$P$ be the toric ideal of $K[x^{v_1},\ldots,x^{v_s}]$ over the field
$K$. By Theorem~\ref{toric-degree-non-homg}, we need only show the
equality
$$
\deg(S/I(\mathcal{L}))= |T(\mathbb{Z}^s/\mathcal{L})|\deg(S/P).
$$ 

Let $\mathcal{L}_s=\{a\in\mathbb{Z}^s\vert\,
{\eta}a\in\mathcal{L}\mbox{ for some
}\eta\in\mathbb{Z}\setminus\{0\}\}$ be the saturation of
$\mathcal{L}$. By part (b), $\mathcal{L}_s$ is equal to ${\rm
  ker}_\mathbb{Z}(A)$. We set
$c=|T(\mathbb{Z}^s/\mathcal{L})|$. Notice that
$T(\mathbb{Z}^s/\mathcal{L})=\mathcal{L}_s/\mathcal{L}$. Recall that a
{\it partial character\/} of $\mathcal{L}_s$ is a homomorphism from
the additive group $\mathcal{L}_s$ to the multiplicative group
$K^*=K\setminus\{0\}$.  According to \cite[Corollaries~2.2 and
  2.5]{EisStu}, there exist distinct partial characters
$\rho_1,\ldots,\rho_c$ of $\mathcal{L}_s$, extending the trivial
character $\rho(a)=1$ for $a\in\mathcal{L}$, such that the minimal
primary decomposition of $I(\mathcal{L})$ is given by
$$  
I(\mathcal{L})=I_{\rho_1}(\mathcal{L}_s)\cap \cdots
\cap I_{\rho_c}(\mathcal{L}_s),
$$
where $I_{\rho_i}(\mathcal{L}_s)$ is a prime ideal generated by all
$t^{a^+}-\rho_i(a)t^{a^-}$ with $a\in\mathcal{L}_s$. As $P$ is a
minimal prime of $I(\mathcal{L})$ by part (b), we may assume 
$\rho_1(a)=1$ for $a\in\mathcal{L}_s$, i.e.,
$P=I_{\rho_1}(\mathcal{L}_s)$. By the additivity of the degree (see
Proposition~\ref{additivity-of-the-degree}), we get 
$$ 
\deg(S/I(\mathcal{L}))=\deg(S/I_{\rho_1}(\mathcal{L}_s))+\cdots+
\deg(S/I_{\rho_c}(\mathcal{L}_s)).
$$
Therefore, it suffices to show that
$\deg(S/P)=\deg(S/I_{\rho_k}(\mathcal{L}_s))$ for  $k=1,\ldots,c$. 
The ideal $I_{\rho_k}(\mathcal{L}_s)$
contains no monomials because it is a prime ideal. Let
$\succ$ be the GRevLex order on $S$ and let $G=\{g_1,\ldots,g_m\}$ be
the reduced Gr\"obner basis of 
$P$ with respect to $\succ$. We can write $g_i=t^{a_i^+}-t^{a_i^-}$
for $i=1,\ldots,m$, with ${\rm in}(g_i)=t^{a_i^+}$. We set
$g_i'=t^{a_i^+}-\rho_{k}(a_i)t^{a_i^-}$ and
$G_k=\{g_1',\ldots,g_m'\}$. Next, we show that $G_k$ is a Gr\"obner basis
of $I_{\rho_k}(\mathcal{L}_s)$. Let $f\neq 0$ be an arbitrary 
(non-necessarily pure) binomial of $I_{\rho_k}(\mathcal{L}_s)$. We
claim that $f$ reduces to zero with respect to $G_k$ in the sense 
of \cite[p.~23]{Ene-Herzog}. By the Division Algorithm (see
\cite[Theorem~2.11]{Ene-Herzog}), we can write 
$$
f=h_1g_1'+\cdots+h_mg_m'+g,
$$
where ${\rm in}(f)\succ {\rm in}(h_ig_i')$ for all $i$, and $g$ is a
binomial (non-necessarily pure) in $I_{\rho_k}(\mathcal{L}_s)$ such that none of
the two terms of $g$ is divisible by any of the monomials
$t^{a_1^+},\ldots,t^{a_m^+}$. We can write  
$$g=\mu(t^a-\lambda t^b)=\mu t^{\delta}(t^{u^+}-\lambda t^{u^-}),$$ 
with $\mu,\lambda\in K^*$ and $a-b=u^+-u^-$. As $t^{u^+}-\lambda t^{u^-}$ is in
$I_{\rho_k}(\mathcal{L}_s)$, it can be seen that $u=u^+-u^-$ is in
$\mathcal{L}_s$. Since $t^{u^+}-\rho_k(u)t^{u^-}\in
I_{\rho_k}(\mathcal{L}_s)$, we get that $\lambda=\rho_k(u)$. If $g\neq 0$,
we obtain that $t^{u^+}-t^{u^-}$, being in $P$, has one of its terms in
the ideal ${\rm in}(P)=(t^{a_1^+},\ldots,t^{a_m^+})$, a
contradiction. Thus, $g$ must be zero, i.e., $f$ reduces to zero with
respect to $G_k$. This proves the claim. In particular, we obtain that
$I_{\rho_k}(\mathcal{L}_s)$ is generated by $G_k$. To show that $G_k$
is a Gr\"obner basis, note that the S-polynomial of $g_i'$ and
$g_j'$ is a binomial; thus, by the claim, it reduces to zero with
respect to $G_k$. 
Therefore by Buchberger's criterion \cite[Theorem~6, p.~84]{CLO}, $G_k$
is a Gr\"obner basis of $I_{\rho_k}(\mathcal{L}_s)$. 
Hence, $S/P$ and $S/I_{\rho_k}(\mathcal{L}_s)$
have the same degree.
\end{proof}

\begin{remark} Part (b) is well-known. The proof given here is
constructive and can be used---in one of the steps---to compute the
degree of an arbitrary lattice ideal (see Example~\ref{nov25-12}).
\end{remark}

The program {\it Normaliz} \cite{normaliz2} computes {\it normalized volumes\/} 
of lattice polytopes using polyhedral geometry. Thus we can compute
the degree of any lattice ideal using
Theorem~\ref{degree-lattice-char=any}. Next we give some applications
and present some examples.

\begin{corollary}\label{oct17-12} 
If $d_1,\ldots,d_s$ are positive integers and $I$
is the toric 
ideal of $K[x_1^{d_1},\ldots,x_1^{d_s}]$, 
then 
$$
\gcd(d_1,\ldots,d_s)\deg(S/I)=\max\{d_1,\ldots,d_s\}.
$$
\end{corollary}

\begin{proof} We may assume that $d_1\leq\cdots\leq d_s$. The order of
$T(\mathbb{Z}/\langle d_1,\ldots,d_s\rangle)$ is
$\gcd(d_1,\ldots,d_s)$. Then, by  
Theorem~\ref{toric-degree-non-homg}, we get that
$\gcd(d_1,\ldots,d_s)\deg(S/I)$ is  
${\rm vol}([0,d_s])=d_s$. 
\end{proof}

\begin{corollary}\label{nov28-12} If $I(\mathcal{L})$ is a lattice
ideal of dimension 
$1$ which is homogeneous with respect to a positive vector $(d_1,\ldots,d_s)$,
then
$$
\gcd(d_1,\ldots,d_s)\deg(S/I(\mathcal{L}))=
\max\{d_1,\ldots,d_s\}|T(\mathbb{Z}^s/\mathcal{L})|.
$$
\end{corollary}

\begin{proof} Let $A$ be the $1\times s$ matrix $(d_1,\ldots,d_s)$. By
hypothesis $\mathcal{L}\subset{\rm ker}_\mathbb{Z}(A)$. The order of
$T(\mathbb{Z}/\langle d_1,\ldots,d_s\rangle)$ is
$\gcd(d_1,\ldots,d_s)$. Then, by
Theorem~\ref{degree-lattice-char=any}, we
get
$$
\deg(S/I(\mathcal{L}))=\frac{T(\mathbb{Z}^s/\mathcal{L}){\rm vol}({\rm
conv}(0,d_1,\ldots,d_s))}{T(\mathbb{Z}/\langle d_1,\ldots,d_s\rangle)}=\frac{
|T(\mathbb{Z}^s/\mathcal{L})|\max\{d_1,\ldots,d_s\}}{\gcd(d_1,\ldots,d_s)},
$$
as required.
\end{proof}

For $1$-dimensional lattice ideals that are not necessarily
homogeneous, we can express the degree in terms of a basis of the
lattice (see Example~\ref{jan10-13}).

\begin{corollary}\label{inspired-by-francesc-example} 
Let $\mathcal{L}\subset\mathbb{Z}^s$ be a lattice of rank $s-1$ and let
$\alpha_1,\ldots,\alpha_{s-1}$ be a $\mathbb{Z}$-basis of
$\mathcal{L}$. If $\alpha_i=(\alpha_{1,i},\ldots,\alpha_{s,i})$, for
$i=1,\ldots,s-1$, and 
\begin{small}
$$ 
v_i=(-1)^i\det\left(
\begin{matrix}
\alpha_{1,1}&\ldots &\alpha_{1,s-1}\\ 
\vdots& &\vdots \\
\alpha_{i-1,1}&\ldots &\alpha_{i-1,s-1}\\ 
\alpha_{i+1,1}&\ldots &\alpha_{i+1,s-1}\\ 
\vdots& &\vdots \\
\alpha_{s,1}&\ldots&\alpha_{s,s-1}
\end{matrix}
\right)\ \mbox{, for}\ i=1,\ldots,s,
$$
\end{small}
then
$\deg(S/I(\mathcal{L}))=\max\{v_1,\ldots,v_s,0\}-\min\{v_1,\ldots,v_s,0\}$.
\end{corollary}

\begin{proof} Let $B$ be the $s\times (s-1)$ matrix with columns
$\alpha_1,\ldots,\alpha_{s-1}$ and let $A$ be the $1\times s$ matrix
  $(v_1,\ldots,v_s)$. We set $r=s-1$. The order of
  $T(\mathbb{Z}^s/\mathcal{L})$ is equal to $\gcd(v_1,\ldots,v_s)$,
  the $\gcd$ of the $r\times r$ minors of $B$. The order of
  $T(\mathbb{Z}/\langle v_1,\ldots,v_s\rangle)$ is also equal to
  $\gcd(v_1,\ldots,v_s)$. Since $AB=0$, we obtain that
  $\mathcal{L}\subset{\ker(A)}$. Hence, by
  Theorem~\ref{degree-lattice-char=any}, we get that
  $\deg(S/I(\mathcal{L}))$ is equal to ${\rm vol}({\rm
    conv}(0,v_1,\ldots,v_s))$ which is equal to
  $\max\{v_1,\ldots,v_s,0\}-\min\{v_1,\ldots,v_s,0\}$.
\end{proof}

The next examples illustrates how to use
Theorems~\ref{toric-degree-non-homg} and 
\ref{degree-lattice-char=any} to compute the degree. 

\begin{example}\label{oct10-12} Let $P$ be the toric ideal of the
monomial subring  
$$
K[H]=K[x_2x_1^{-1},x_3x_2^{-1},x_4x_3^{-1},
x_1x_4^{-1},x_5x_2^{-1},x_3x_5^{-1},x_4x_5^{-1}].
$$
We employ the notation of Theorem~\ref{toric-degree-non-homg}. 
The rank of $A$ is $4$ and the height of $P$ is $3$. Using {\it
Normaliz}, we get 
$$
4!{\rm vol}({\rm conv}(v_1,\ldots,v_7,0))=11.
$$
As the group $\mathbb{Z}^5/\langle v_1,\ldots,v_7\rangle$ is torsion 
free, by Theorem~\ref{toric-degree-non-homg} we have 
${\rm deg}(K[t_1,\ldots,t_7]/P)=11$.
\end{example}

\begin{example}\label{nov25-12} 
Let $K$ be the field of rational numbers and let $\mathcal{L}$ be the
lattice generated by  
\begin{eqnarray*}
a_1=(2,1,1,1,-1,-1,-1,-2),& &a_2=(1,1,-1,-1,1,1,-1,-1),\\ 
a_3=(2,-1,1,-2,1,-1,1,-1),& &
a_4=(5,-5,0,0,0,0,0,0).
\end{eqnarray*}
We use the notation of Theorem~\ref{degree-lattice-char=any}. The
lattice $\mathcal{L}$ has rank $4$ and 
 $a_1,\ldots,a_4,e_1,e_3,e_4,e_5$ form a $\mathbb{Q}$-basis of
 $\mathbb{Q}^8$. In this case we obtain the matrix 
$$
A=\left(\begin{array}{rrrrrrrr}
4& 4& 0& 0& 0& -1& 1& 6\\ 
0& 0& 1& 0& 0& 1& 0& 0\\
 0& 0& 0& 4& 0& 7& 9& -6\\
 0& 0& 0& 0& 2& -3& -3& 2
\end{array}\right)
$$
whose columns are denoted by $v_1,\ldots,v_8$. Let $P$ be the toric
ideal of $K[x^{v_1},\ldots,x^{v_8}]$. Therefore, by
Theorem~\ref{degree-lattice-char=any}, we get 
$$
\deg(S/I(\mathcal{L}))=|T(\mathbb{Z}^8/\mathcal{L})|\deg(S/P)=
\frac{5(4)!{\rm
vol}({\rm conv}(0,v_1,\ldots,v_8))}{|T(\mathbb{Z}^{4}/\langle
v_1,\ldots,v_8\rangle)|}=\frac{(5)(200)}{8}=125.
$$
The normalized volume of the polytope ${\rm conv}(0,v_1,\ldots,v_8)$
was computed using {\it Normaliz} \cite{normaliz2}. 
\end{example}

\section{Primary decompositions of 
homogeneous binomial ideals}\label{primary-dec-1}

In this part we present some applications of the main result of
Section~\ref{prim-dec-deg-tor} to the theory of graded binomial ideals
and graded lattice ideals of dimension $1$. We continue to employ the
notations and definitions used in Section~\ref{prelim}.

Given a subset $I\subset S$, its {\it variety\/}, denoted by $V(I)$,
is the set of all  
$a\in\mathbb{A}_K^s$ such that $f(a)=0$ for all $f\in I$, where
$\mathbb{A}_K^s$ denotes affine $s$-space over $K$. In this 
section we focus on homogeneous binomial ideals $I$ of $S$ 
with $V(I,t_i)=\{0\}$ for all $i$.

\begin{lemma}\label{sep16-12}
Let $I$ be a homogeneous binomial ideal of $S$ such that $V(I,t_i)=\{0\}$ for
all $i$. Then, for any associated prime $\mathfrak{p}$ of $I$, 
either $\mathfrak{p}=\mathfrak{m}$, or else
$\mathfrak{p}\subsetneq \mathfrak{m}$, $t_i\not\in\mathfrak{p}$ for all $i$ and ${\rm
ht}(\mathfrak{p})=s-1$. 
Moreover, ${\rm ht}(I)=s-1$.
\end{lemma}
\begin{proof}
Since $I$ is homogeneous, any associated prime $\mathfrak{p}$ of $I$ is
homogeneous too. Hence $I\subset\mathfrak{p}\subset\mathfrak{m}$. In particular,
$(I,t_1)\subset \mathfrak{m}$. Let $\fq$ be any minimal prime ideal over
$(I,t_1)$. By \cite[Lemma~2.6]{ci-lattice}, $\fq=\mathfrak{m}$ and so
$s={\rm ht}(I,t_1)\leq {\rm ht}(I)+1$ (here we use the fact
that $I$ is homogeneous). Thus ${\rm ht}(I)\geq s-1$. Suppose that
$\mathfrak{p}$ is an associated prime of $I$, $\mathfrak{p}\neq\mathfrak{m}$.  Then
$t_i\not\in\mathfrak{p}$ for all $i$, because if 
some $t_i\in\mathfrak{p}$, $(I,t_i)\subset\mathfrak{p}$, and
by \cite[Lemma~2.6]{ci-lattice} again, $\mathfrak{p}$ would be equal to
$\mathfrak{m}$. In particular, ${\rm ht}(\mathfrak{p})=s-1$.  Finally ${\rm ht}(I)=s-1$,
otherwise $\mathfrak{m}$ would be the only associated prime of $I$. Thus
$\mathfrak{m}=\rad(I)$, a contradiction because $I$ cannot contain a power of
$t_i$ for any $i=1,\ldots ,s$.
\end{proof}

The assumption that $I$ is homogeneous in the result above is crucial.

\begin{example} Let $S=K[t_1,t_2,t_3]$ and
$I=(t_1t_2t_3-1)\mathfrak{m}$. Clearly $V(I,t_i)=\{0\}$ for all $i$.
However, ${\rm ht}(I)=1$ and ${\rm ht}(I,t_i)=3$ for all $i$.
\end{example}

An ideal $I$ is said to be {\it unmixed\/} if all of its associated primes have
the same height (see, e.g., \cite[p.~196]{ZS}). We write $\hull(I)$
for the intersection of the 
isolated primary components of $I$. 

\begin{proposition}\label{2variables}
Let $I$ be a graded binomial ideal of $S$ such that
$V(I,t_i)=\{0\}$ for all $i$. The following conditions are equivalent:
\begin{eqnarray*}
&&\mbox{$(a)$ $I$ is a lattice ideal;}\phantom{+}\mbox{$(b)$
    $I=(I:(t_1\cdots t_s)^{\infty})$;}\phantom{+}\mbox{$(c)$ $t_i$ is
    regular modulo $I$ for all $i$;}\\ &&\mbox{$(d)$
    $I=(I:t_1)$;}\phantom{+}\mbox{$(e)$ $I$ is
    Cohen-Macaulay;}\phantom{+}\mbox{$(f)$ $I$ is
    unmixed;}\phantom{+}\mbox{$(g)$ $I=\hull(I)$.} 
\end{eqnarray*}
\end{proposition}
\begin{proof}
The equivalences among $(a)$, $(b)$ and $(c)$ follow from
Theorem~\ref{jun12-02}. By Lemma~\ref{sep16-12}, $(f)$ and $(g)$ are
equivalent. Clearly $(c)\Rightarrow (d)$. By
Lemma~\ref{sep16-12}, ${\rm ht}(I)=s-1$ and, for any associated prime
$\mathfrak{p}$ of $I$, either $\mathfrak{p}=\mathfrak{m}$, 
or else $\mathfrak{p}\subsetneq \mathfrak{m}$,
${\rm ht}(\mathfrak{p})=s-1$ and $t_1\cdots t_s\not\in\mathfrak{p}$.
Therefore $I$ has 
a minimal primary decomposition either of the form
$I=\fq_1\cap\cdots\cap \fq_c=\hull(I)$, or else $I=\fq_1\cap\cdots\cap
\fq_c\cap \fq=\hull(I)\cap \fq$, where the $\fq_i$ are $\mathfrak{p}_i$-primary
ideals with ${\rm ht}(\mathfrak{p}_i)=s-1$, and $\fq$ is
$\mathfrak{m}$-primary. Therefore $(e)$ and $(f)$ are equivalent. Moreover,
either
$(I:t_1^{\infty})=\cap_{i=1}^{c}(\fq_i:t_1^{\infty})=\cap_{i=1}^{c}\fq_i=\hull(I)$,
or else $(I:t_1^{\infty})=\cap_{i=1}^{c}(\fq_i:t_1^{\infty})\cap
(\fq:t_1^{\infty})=\cap_{i=1}^{c}\fq_i=\hull(I)$, since
$(\fq:t_1^{\infty})=S$. In both cases, $\hull(I)=(I:t_1^{\infty})$.

Suppose $(d)$ holds. Then $I=(I:t_1)=(I:t_1^{\infty})=\hull(I)$ and $I$ is
unmixed. Thus $(d)\Rightarrow (f)$. Suppose $(f)$ holds. Then $\mathfrak{m}$ is
not an associated prime of $I$, and, by Lemma~\ref{sep16-12}, $t_i$
cannot be in any associated prime of $I$, so each $t_i$ is regular
modulo $I$ and $(c)$ holds. 
\end{proof}

\begin{definition}\label{matrix-ideal}
Let $L$ be an $s\times m$ integer matrix, let $l_{*,i}$ be the 
$i$-th column of $L$ and let $f_i=t^{l_{*,i}^+}-t^{l_{*,i}^-}$ be the
binomial of $S=K[t_1,\ldots,t_s]$ defined by $l_{*,i}$. 
We call $I(L):=(f_1,\ldots ,f_m)$ the {\it matrix ideal\/} associated
to $L$. 
\end{definition}

Whenever $I(L)$ is graded with respect to a weight vector
$\mathbf{d}=(d_1,\ldots,d_s)\in\mathbb{N}_+^s$, we can
and will suppose, without loss of generality, that $\gcd(\mathbf{d})=1$. 

\begin{definition}\cite[p.~397]{opPCB}\label{herzogideal}
Let $\mathbf{d}=(d_{1},\ldots ,d_{s})\in\bn^{s}_{+}$. The {\em Herzog
  ideal associated to $\mathbf{d}$}, denote by
$\mathfrak{p}_{\mathbf{d}}$, is the toric ideal of
$K[x_1^{d_1},\ldots,x_1^{d_s}]$, where $x_1$ is a variable.
\end{definition}

\begin{remark}\label{gradings}
Let $I(L)=(f_1,\ldots ,f_m)$ be the matrix ideal associated to an
$s\times m$ integer matrix $L$. The following conditions are
equivalent.
\begin{itemize}
\item[$(a)$] There exists an $\bn$-grading of $S$, with each $t_i$
  of weight $d_i>0$, under which each $f_i$ is homogeneous of
  positive degree;
\item[$(b)$] There exists $\mathbf{d}=(d_1,\ldots,d_s)\in\bn^s_{+}$ such that
  $\mathbf{d} L=0$;
\item[$(c)$] There exists $\mathbf{d} =(d_1,\ldots,d_s)\in\bn^s_{+}$ with
  $I(L)\subset \mathfrak{p}_{\mathbf{d}}$, the Herzog ideal
  associated to $\mathbf{d}$. 
\end{itemize}
\end{remark}

We give now a result on the structure of graded matrix ideals.

\begin{proposition}\label{structureI(L)}
Let $I$ be the matrix ideal of an $s\times m$ integer
matrix $L$ and let $\bl$ be the lattice spanned by the columns
of $L$. Suppose that $I$ is graded and that $V(I,t_i)=\{0\}$ for
all $i$. Then:
\begin{itemize}
\item[$(a)$] $I$ has a minimal primary decomposition either of the
  form $I=\fq_1\cap\cdots\cap \fq_c$, if $I$ is unmixed, or else
  $I=\fq_1\cap\cdots\cap \fq_c\cap\fq$, if $I$ is not unmixed, where
  the $\fq_i$ are $\mathfrak{p}_i$-primary ideals 
with ${\rm ht}(\mathfrak{p}_i)=s-1$,
  and $\fq$ is an $\mathfrak{m}$-primary ideal.
\item[$(b)$] For all
  $g\in\mathfrak{m}\setminus\cup_{i=1}^{c}\mathfrak{p}_i$, 
$I(\bl)=(I:(t_1\cdots t_s)^{\infty})=(I:g^{\infty})=\fq_1\cap\cdots\cap
  \fq_c=\hull(I)$. 
\item[$(c)$] $\rank(L)=s-1$ and there exists a unique Herzog ideal
$\mathfrak{p}_{\mathbf{d}}$ 
containing $I$.
\item[$(d)$] If $I$ is not unmixed and $h=t_1\cdots t_s$, there
exists $a\in\bn_+$ such that 
  $I(\bl)=(I\colon h^{a})$, $\fq=I+(h^a)$ is an irredundant
  $\mathfrak{m}$-primary 
  component of $I$ and $I=I(\bl)\cap \fq$.
\item[$(e)$] Either $c\leq |T(\mathbb{Z}^s/\mathcal{L})|$, if
  $\ch(K)=0$, or else $c\leq |G|$, if $\ch(K)=p$, $p$ a prime, where $G$
  is the unique largest subgroup of $T(\mathbb{Z}^s/\mathcal{L})$
  whose order is relatively prime to $p$. If $K$ is algebraically
  closed, then equality holds.
\end{itemize}
\end{proposition}

\begin{proof} We set $h=t_1\cdots t_s$. Item $(a)$ follows 
from Lemma~\ref{sep16-12} (see the proof of
Proposition~\ref{2variables}). By Lemma~\ref{sep1-12},
$I(\bl)=(I\colon h^{\infty})$, which proves the first equality in
$(b)$. Let $g\in\mathfrak{m}\setminus\cup_{i=1}^{c}\mathfrak{p}_i$. 
Suppose that $I=\fq_1\cap\cdots\cap \fq_c$. Then, for any
$a\in\bn_{+}$, we have
$(I\colon g^a)=\cap_{j=1}^{c}(\fq_j\colon g^a)=\cap_{j=1}^{c}\fq_j$, because
$g^a\not\in\mathfrak{p}_j$ and $\fq_j$ is $\mathfrak{p}_j$-primary. Suppose that
$I=\fq_1\cap\cdots\cap \fq_c\cap \fq$. Then, for $a\gg 0$, $(I\colon
g^a)=(I\colon g^\infty)$ and 
$g^a\in\mathfrak{m}^a\subset\fq$, so that $(\fq\colon g^a)=S$. Thus
$(I\colon g^a)=\cap_{j=1}^{c}(\fq_j\colon g^a)\cap
(\fq\colon g^a)=\cap_{j=1}^{c}\fq_j$. In either case, one has the equality 
$(I\colon g^{\infty})=\fq_1\cap\cdots\cap\fq_c$, which
coincides with $\hull(I)$. By a similar calculation, we get 
$(I\colon h^\infty)=\fq_1\cap\cdots\cap\fq_c$. 

In particular, using $(b)$, ${\rm ht}(I)={\rm ht}(I(\bl))$. Since
${\rm ht}(I(\bl))=\rank(\bl)=\rank(L)$, 
 it follows that $\rank(L)=s-1$.
Since $I$ is homogeneous, by Remark~\ref{gradings}, there exists
$\mathbf{d}\in\bn^{s}_{+}$ such that $I\subset\mathfrak{p}_{\mathbf{d}}$, the
Herzog ideal associated to $\mathbf{d}$. Since
$\rank(L^{\top})=\rank(L)=s-1$, then $\ker(L^{\top})$ is
generated as a $\bq$-linear subspace by $\mathbf{d}^{\top}$. 
Thus $\mathfrak{p}_{\mathbf{d}}$
is the unique Herzog ideal containing $I$ (see
\cite[Remark~3.2]{opPCB}). This proves $(c)$.

Suppose that $I$ is not unmixed. As $S$ is a Noetherian ring, there
exists $a\in\bn_{+}$ such that 
$(I\colon h^a)=(I\colon h^{\infty})$. By \cite[Proposition~7.2]{EisStu},
$I=(I\colon h^a)\cap (I+(h^a))$, where 
$(I\colon h^a)=I(\bl)$ is
the hull of $I$. Since $I$ is not unmixed, $I+(h^a)$ must be
irredundant. Moreover, by \cite[Lemma~2.6]{ci-lattice}, the only prime
ideal containing $I+(h^a)$ is $\mathfrak{m}$. It follows that $I+(h^a)$ is
$\mathfrak{m}$-primary. This proves $(d)$. Finally, $(e)$ follows from
Theorem~\ref{bounds-for-the-number-of-associated-primes}(b).
\end{proof}

Let $I$ be the matrix ideal associated to an $s\times m$
integer matrix $L$. The conditions $I$ homogeneous and $\rank(L)=s-1$
do not imply that $V(I,t_i)=\{0\}$ for all $i$. 

\begin{example} $L$ be the matrix with column vectors $(2,-2,0)$ and
$(-2,1,1)$ and let $I$ be its matrix ideal
$(t_1^2-t_2^2,t_1^2-t_2t_3)$. Clearly $I$ is 
homogeneous (with the standard grading) and $\rank(L)=2$. However
$(0,0,\lambda)\in V(I,t_1)$ for any $\lambda\in K$.
\end{example}

As a consequence of the previous result, we
obtain the following.

\begin{corollary}\label{degreeI(L)}
Let $I=I(L)$ be the matrix ideal associated to an $s\times m$
integer matrix $L$. Let $\bl$ be the lattice spanned by the columns
of $L$. Suppose that $I$ is homogeneous with respect to a weight vector
$\mathbf{d}=(d_1,\ldots ,d_s)\in\bn^s_+$ and that
$V(I,t_i)=\{0\}$ for all $i$. Then
\begin{eqnarray*}
\deg(S/I)=\max\{d_1,\ldots,d_s\}|T(\bz^s/\bl)|
=\max\{d_1,\ldots,d_s\}\Delta_{s-1}(L).
\end{eqnarray*}
\end{corollary}
\begin{proof}
By Proposition~\ref{structureI(L)},
$\hull(I)=\fq_1\cap\cdots\cap\fq_c=I(\bl)$. Thus, by
Proposition~\ref{additivity-of-the-degree},
$$
\deg(S/I)=\textstyle\sum_{j=1}^{c}\deg(S/\fq_j)=\deg(S/I(\bl)).
$$
By
Corollary~\ref{nov28-12}, we have $\deg(S/I(\bl))=
\max\{d_1,\ldots,d_s\}|T(\bz^s/\bl)|$.  To complete the proof
notice that, by Theorem~\ref{torsion}, one has
$|T(\bz^s/\bl)|=\Delta_{s-1}(L)$ .
\end{proof}

\section{Generalized positive critical binomial ideals
}\label{gpcb-section}

We continue to employ the notations and 
definitions used in Section~\ref{intro}. In this section we restrict
to the study of certain classes of square integer matrices (see
Definition~\ref{square-matrix-new}) and their corresponding matrix
ideals.  

Recall that the {\it support\/} of a polynomial $f\in S$, denoted by
${\rm supp}(f)$, is defined as the set of all variables $t_i$ that
occur in $f$.

\begin{proposition}\label{oct6-12} 
Let $I=I(L)=(f_1,\ldots,f_s)$ be the matrix ideal associated to an
$s\times s$ PB matrix $L$. Suppose that $|{\rm supp}(f_j)|\geq 4$, for
all $j=1,\ldots ,s$. Then:
\begin{itemize}
\item[$(a)$] $I$ is not a lattice ideal.
\item[$(b)$] If $V(I,t_i)=\{0\}$ for all $i$, then $\mathfrak{m}$ is
  an associated prime of $I$.
\item[$(c)$] If $I$ is graded and $V(I,t_i)=\{0\}$ for all $i$,  
then $I$ has a minimal
  primary decomposition of the form $I=\fq_1\cap\cdots\cap
  \fq_c\cap\fq$, where the $\fq_i$ are $\mathfrak{p}_i$-primary
  ideals with 
  ${\rm ht}(\mathfrak{p}_i)=s-1$, and $\fq=(I,(t_1\cdots t_s)^a)$ 
is an $\mathfrak{m}$-primary ideal,
  for some $a\in\mathbb{N}_+$. Moreover, $I$ has at most
  $|T(\mathbb{Z}^s/\mathcal{L})|+1$ primary components. 
\end{itemize}
\end{proposition}
\begin{proof} 
Assume that $I$ is a lattice ideal. There is $k>1$ such that
$f_1=t_1^{a_{1,1}}-t_k^{a_{k,1}}\cdots t_s^{a_{s,1}}$ and
$f_k=t_k^{a_{k,k}}-t_1^{a_{1,k}}\cdots
t_{k-1}^{a_{k-1,k}}t_{k+1}^{a_{k+1,k}}\cdots t_s^{a_{s,k}}$, with
$a_{k,1}>0$. We claim that $a_{1,1}>a_{1,k}$. If $a_{1,1}\leq
a_{1,k}$, then from the equality
\begin{eqnarray*}
(t_1^{a_{1,1}}-t_k^{a_{k,1}}\cdots
  t_s^{a_{s,1}})t_1^{a_{1,k}-a_{1,1}}t_2^{a_{2,k}}\cdots
  t_{k-1}^{a_{k-1,k}}t_{k+1}^{a_{k+1,k}}\cdots t_s^{a_{s,k}}
  \phantom{++++++}& &\\ +(t_k^{a_{k,k}}-t_1^{a_{1,k}}\cdots
  t_{k-1}^{a_{k-1,k}}t_{k+1}^{a_{k+1,k}}\cdots
  t_s^{a_{s,k}})=\phantom{}&
  &\\ t_k^{a_{k,k}}-t_1^{a_{1,k}-a_{1,1}}t_2^{a_{2,k}}\cdots
  t_{k-1}^{a_{k-1,k}}t_k^{a_{k,1}}t_{k+1}^{a_{k+1,k}+a_{k+1,1}}\cdots
  t_s^{a_{s,k}+a_{s,1}}=:t_kg\neq 0, & &
\end{eqnarray*}
one has $t_kg\in I$. As $t_k$ is a nonzero divisor of $S/I$, we get
$g\in I$. Thus, $g$ is a linear combination with coefficients in $S$
of $f_1,\ldots,f_s$. Since $t_k^{a_{k,k}-1}$ is a term of $g$, we
conclude that $t_k^{a_{k,k}-1}$ is a multiple of some term of $f_i$
for some $i\neq k$, a contradiction to the fact that, a fortiori,
${\rm supp}(f_{i})$ has at least $3$ variables. This proves
$a_{1,1}>a_{1,k}$.  Moreover, from the equality
\begin{eqnarray*}
(t_k^{a_{k,k}}-t_1^{a_{1,k}}\cdots
  t_{k-1}^{a_{k-1,k}}t_{k+1}^{a_{k+1,k}}\cdots
  t_s^{a_{s,k}})t_1^{a_{1,1}-a_{1,k}}\phantom{+++++++++++}& &
  \\ +(t_1^{a_{1,1}}-t_k^{a_{k,1}}\cdots
  t_s^{a_{s,1}})t_2^{a_{2,k}}\cdots
  t_{k-1}^{a_{k-1,k}}t_{k+1}^{a_{k+1,k}}\cdots t_s^{a_{s,k}}=
  \phantom{}& &\\ t_k^{a_{k,k}}t_1^{a_{1,1}-a_{1,k}}-t_2^{a_{2,k}}\cdots
  t_{k-1}^{a_{k-1,k}}t_k^{a_{k,1}}t_{k+1}^{a_{k+1,k}+a_{k+1,1}}\cdots
  t_s^{a_{s,k}+a_{s,1}}=:t_kg'\neq 0,& &
\end{eqnarray*}
the argument above shows that $t_k^{a_{k,k}-1}t_1^{a_{1,1}-a_{1,k}}$
is a multiple of some term of $f_i$ for some $i$, a contradiction to
the fact that ${\rm supp}(f_{i})$ has at least $4$ variables. Hence
$I$ is not a lattice ideal.

Since $I$ is not a lattice ideal, $t_i$ is a zero divisor of $S/I$ for
some $i$. Then some associated prime $\mathfrak{p}$ of $I$ contains
$t_{i}$. Hence, by \cite[Lemma~2.6]{ci-lattice}, $\mathfrak{p}$ is
equal to $\mathfrak{m}$, which proves $(b)$. Since $I$ is not a
lattice ideal, $I$ is not unmixed (see
Proposition~\ref{2variables}). The rest follows on applying
Proposition~\ref{structureI(L)}.
\end{proof}

\begin{proposition}\label{feb9-13} Let $g_1,\ldots,g_s$ be the binomials defined by
the rows of a GCB matrix $L$ and let $I$ be the ideal generated
by $g_1,\ldots,g_s$. If $V(I,t_i)=\{0\}$ and $|{\rm supp}(g_i)|\geq 3$
for all $i$, then $I$ is not a complete intersection. 
\end{proposition}

\begin{proof} We may assume that $g_i$ corresponds to the
$i$-th row of $L$. We proceed by contradiction.
Assume that $I$ is a complete intersection. As $I$ is graded, since
$L$ is a GCB matrix, and 
the height of $I$ is $s-1$, we may assume that $g_1,\ldots,g_{s-1}$ 
generate $I$. Hence, we can write
$$
g_s=t_s^{a_{s,s}}-t_1^{a_{s,1}}\cdots t_{s-1}^{a_{s,s-1}}=h_1g_1+\cdots+h_{s-1}g_{s-1}
$$
for some $h_1,\ldots,h_{s-1}$ in $S$, where $a_{s,s}>0$. Therefore, the monomial
$t_s^{a_{s,s}}$ has to occur in the right hand side of this equation,
a contradiction
to the fact that $|{\rm supp}(g_i)|\geq 3$ for all $i$.
\end{proof}

In what follows, we examine GPCB matrices---a natural extension
of the PCB matrices introduced in \cite{opPCB}---and the algebra of
their matrix ideals. The class of GPCB matrices is closed under
transposition, as the next result shows. 

\begin{theorem}\label{GPCB-new} Let $L$ be an integer matrix of size
 $s\times s$ with rows $\ell_1,\ldots,\ell_s$ and let $(L_{i,j})$ be
  the adjoint matrix of $L$. Suppose $L\mathbf{b}^\top=0$ for some
  $\mathbf{b}$ in $\mathbb{N}_+^s$. The following hold.
\begin{itemize}
\item[\rm(a)] If ${\rm
rank}(\ell_1,\ldots,\widehat{\ell}_i,\ldots,\ell_s)=s-1$ and 
$L_{i,i}\geq 0$ for all $i$,
then $\mathbf{c}L=0$ for some $\mathbf{c}\in\mathbb{N}_+^s$. 

\item[\rm(b)] If $L_{i,i}>0$ for all $i$, then $\mathbf{c}L=0$ for some
$\mathbf{c}\in\mathbb{N}_+^s$. 

\item[\rm(c)] If $L$ is a GCB matrix and ${\rm
rank}(\ell_1,\ldots,\widehat{\ell}_i,\ldots,\ell_s)=s-1$ for all $i$,
then $L^\top$ is a GCB matrix.  

\item[\rm(d)] If $L$ is a GPCB matrix, then ${\rm rank}(L)=s-1$ and
$L^{\top}$ is a GPCB 
matrix. 

\item[\rm(e)] If $s=3$ and $L$ is a GCB matrix, then ${\rm rank}(L)=2$
and $L^{\top}$ is a 
GCB matrix. 
\end{itemize}
\end{theorem}

\begin{proof} (a): Let $L_i$ be the $i$-th column of 
${\rm adj}(L)=(L_{i,j})$. Since
$\ell_1,\ldots,\widehat{\ell}_i,\ldots,\ell_s$ are linearly
independent, we get $L_i\neq 0$. The vector $\mathbf{b}$ generates 
${\rm ker}_\mathbb{Q}(L)$ because $L$ has rank $s-1$ and
$L\mathbf{b}^\top=0$. Then,
because of the equality $L{\rm adj}(L)=0$, we can write $L_i=\mu_i\mathbf{b}$
for some $\mu_i\in\mathbb{Q}$. Notice that $\mu_i>0$ because
$L_{i,i}\geq 0$ and $\mathbf{b}\in\mathbb{N}_+^s$. Hence, all entries of ${\rm
adj}(L)$ are positive integers. If $\mathbf{c}$ is any row of 
${\rm adj}(L)$, we get $\mathbf{c}L=0$ because ${\rm
adj}(L)L=0$. 

(b): For any $i$, the vectors are
$\ell_1,\ldots,\widehat{\ell}_i,\ldots,\ell_s$ are linearly 
independent because $L_{i,i}>0$. Thus this part follows from (a). 

 (c): Let $L$ be a GCB matrix as in
Definition~\ref{square-matrix-new}, Eq.~(\ref{gcb-matrix}).
($\mathrm{c}_1$): First we treat the case
$\mathbf{b}=\mathbf{1}=(1,\ldots,1)$, i.e., the case where $L$ is a CB
matrix. By part (a) it suffices to show that $L_{i,i}\geq 0$ for all
$i$. Let $H_{i,i}$ be the submatrix of $L$ obtained by eliminating the
$i$-th row and $i$-th column. By the Gershgorin Circle Theorem, every
(possibly complex) eigenvalue $\lambda$ of $H_{i,i}$ lies within at
least one of the discs $\{z\in\mathbb{C}\vert\, \| z -a_{j,j}\|\leq
r_j\}$, $j\neq i$, where $r_j=\sum_{u\neq i,j}|-a_{j,u}|\leq a_{j,j}$
since $L\mathbf{1}^\top=0$ and $a_{i,j}\geq 0$ for all $i,j$. If
$\lambda\in\mathbb{R}$, we get $|\lambda-a_{j,j}|\leq a_{j,j}$, and
consequently $\lambda \geq 0$. If $\lambda\notin\mathbb{R}$, then
since $H_{i,i}$ is a real matrix, its conjugate $\overline{\lambda}$
must also be an eigenvalue of $H_{i,i}$. Since $\det(H_{i,i})$ is the
product of the $s-1$ (possibly repeated) eigenvalues of $H_{i,i}$, we
get $L_{i,i}\geq 0$. This argument is adapted from the proof of
\cite[Lemma 2.1]{opPCB}. ($\mathrm{c}_2$): Now, we treat the general
case. Let $B$ be the $s\times s$ diagonal matrix
$\diag(b_1,\ldots,b_s)$, where $\mathbf{b}=(b_1,\ldots,b_s)$, and let
$\widetilde{L}=LB$. Notice that $\widetilde{L}\mathbf{1}^\top=0$
because $L\mathbf{b}^\top=0$, and $\widetilde{L}$ is a CB matrix
because $L$ is a GCB matrix. Let $(\widetilde{L}_{i,j})$ be the
adjoint matrix of $\widetilde{L}$. Since $b_i>0$ for all $i$, by the
multilinearity of the determinant, it follows that $L_{i,j}\neq 0$ if
and only if $\widetilde{L}_{i,j}\neq 0$. Hence, any set of $s-1$ rows
of $\widetilde{L}$ is linearly independent.  Therefore, applying case
($\mathrm{c}_1$) to $\widetilde{L}$, we obtain that
$\widetilde{L}^\top$ is a GCB matrix. Thus, there is
$\mathbf{c}\in\mathbb{N}_+^s$ such that
$\mathbf{c}\widetilde{L}=0$. Then, $\mathbf{c}L=0$, i.e., $L^\top$ is
a GCB matrix, as required.

(d): Let $L$ be a GPCB matrix. By the argument given in
($\mathrm{c}_1$), it follows readily that $L_{i,i}>0$ 
for all $i$. In particular the rank of $L$ is $s-1$. Hence, by 
part (b), $L^\top$ is a GPCB matrix.  

(e): Let $L$ be a GCB matrix with $s=3$. It is easy to see 
that $L_{i,j}>0$ for $i\neq j$. In particular $L$ has rank $2$ 
and any two rows of $L$ are linearly independent. Then by (c), $L^\top$ is
a GCB matrix. 
\end{proof}

Following the proof of (c) above, we call $\widetilde{L}:=LB$ the
{\it associated\/} CB {\it matrix\/} of $L$.

\begin{lemma}\label{march19-13} Let $L$ be a GCB matrix of size $s\times s$. 
If ${\rm rank}(L)=s-1$, then any set of $s-1$ columns is linearly
independent.
\end{lemma}

\begin{proof} Let $\ell_{*,1},\ldots,\ell_{*,s}$ be the columns of
$L$. By hypothesis, there is
$\mathbf{b}=(b_1,\ldots,b_s)\in\mathbb{N}_+^s$ such that
$L\mathbf{b}^\top=0$. Thus it suffices to observe that 
$b_1\ell_{*,1}+\cdots+b_s\ell_{*,s}=0$. 
\end{proof}

\begin{example} Let $L$ be the $4\times 4$ integer matrix of rank 3:
\begin{small}
\begin{eqnarray*}
L=\left(\begin{array}{rrrr} 5&-2&0&-1\\0&1&-1&0\\0&-1&1&0\\-1&0&-1&4
\end{array}\right)
\end{eqnarray*}
\end{small}
Then $L$ is a GCB matrix, where $L\mathbf{b}^\top=0$ if
$\mathbf{b}=(9, 19, 19, 7)$ but $L^\top$ is not a GCB matrix because
rows $2,3$ and $4$ are linearly dependent (see
Lemma~\ref{march19-13}).
\end{example}

\begin{remark}\label{onecanapply}
Observe that if $I=I(L)$ is a GPCB ideal associated to a GPCB matrix
$L$, then $(I,t_1)=(t_1,t_2^{a_{2,2}},\ldots ,t_s^{a_{s,s}})$, ${\rm
  rad}(I,t_1)=\mathfrak{m}$ and $V(I,t_1)=\{0\}$ (and similarly
$V(I,t_i)=\{0\}$ for all the other variables).  Moreover, by
Theorem~\ref{GPCB-new}, $I$ is homogeneous. Therefore, one can apply
to GPCB ideals most of the results of the previous section.
\end{remark}

In what follows, we extend to GPCB ideals some properties that hold
for PCB ideals. First of all, using Theorem~\ref{GPCB-new}, we get
that \cite[Proposition~3.3(a)--(c)]{opPCB} holds for any GPCB ideal
provided that we assume $s\geq 3$ in part (c). Regarding the
(un)mi\-xed\-ness 
property of GPCB ideals, observe 
that Proposition~\ref{structureI(L)} generalizes
\cite[Proposition~4.1]{opPCB}.

To simplify notations and to avoid repetitions, for the
rest of this section we assume that $L$ is a GPCB matrix with 
${\bf b}=(b_1,\ldots ,b_s)\in\mathbb{N}_{+}^{s}$, $\gcd({\bf b})=1$
and $L{\bf b}^{\top}=0$.  If ${\rm char}(K)=p$, $p$ a prime, since  
$\gcd({\bf b})=1$, then on reordering the
variables, one can always suppose without loss of generality that
$p\nmid b_s$. From now on and until the end of the section, we will 
suppose that, if ${\rm char}(K)=p>0$, then $p\nmid b_s$. The entries
of $L$ are denoted by $a_{i,i}$ and $-a_{i,j}$ if $i\neq j$. 
As in the proof of Theorem~\ref{GPCB-new}(c), the
matrix $\widetilde{L}:=LB,$ where $B:=\ $diag$(b_{1},...,b_{s}),$ denotes
the so-called PCB matrix associated to $L.$

\begin{proposition}\label{explicitrelation} 
$($cf. \cite[Remark~3.4]{opPCB}$)$ Let $I=I(L)=(f_1,\ldots ,f_s)$ be
  the ideal of a GPCB matrix $L$ and let $\widetilde{L}$ be the PCB
  matrix associated to $L$. For $i=1,\ldots ,s$, let
  $x_i=t_{i}^{a_{i,i}}$ and $y_i=t_{1}^{a_{1,i}}\cdots
  t_{i-1}^{a_{i-1,i}}t_{i+1}^{a_{i+1,i}}\cdots t_{s}^{a_{s,i}}$ be
  the two terms of $f_i$, so that $f_i=x_i-y_i$. If $b_i=1$, set
  $g_i=1$ and $q_i=0$. If $b_i>1$, set
  $g_i=\sum_{j=1}^{b_i}x_i^{b_i-j}y_i^{j-1}\in S$. Let
  $q_i=\sum_{j=1}^{b_i-1}jx_i^{b_i-1-j}y_i^{j-1}\in
  S$ and
\begin{eqnarray*}
&&b(1)=(0,0,b_{3,3}-b_{3,4}-\cdots -b_{3,s}-b_{3,1},
  b_{4,4}-b_{4,5}-\cdots -b_{4,s}-b_{4,1},\ldots
  ,b_{s,s}-b_{s,1})\in\mathbb{N}^{s},
  \\ &&b(2)=(b_{1,1}-b_{1,2},0,0,b_{4,4}-b_{4,5}-\cdots
  -b_{4,s}-b_{4,1}-b_{4,2},\ldots ,
  b_{s,s}-b_{s,1}-b_{s,2})\in\mathbb{N}^{s},
  \\&&b(3)=(b_{1,1}-b_{1,2}-b_{1,3},b_{2,2}-b_{2,3},0,0,\ldots
  ,b_{s,s}-b_{s,1}-b_{s,2}-b_{s,3})\in\mathbb{N}^{s}\mbox{, }\ldots
  ,\\ && b(s-1)=(b_{1,1}-b_{1,2}-\cdots -b_{1,s-1},\ldots
  ,b_{s-2,s-2}-b_{s-2,s-1},0,0)\in\mathbb{N}^{s}\mbox{ and } \\&&
  b(s)=(0,b_{2,2}-b_{2,3}-\cdots -b_{2,s},b_{3,3}-b_{3,4}-\cdots
  -b_{3,s},\ldots ,b_{s-1,s-1}-b_{s-1,s},0)\in\mathbb{N}^{s},
\end{eqnarray*}
where $b_{i,i}$ and $-b_{i,j}$, $i\neq j$, are the entries of
$\widetilde{L}$. Then, for each $i=1,\ldots ,s$,
\begin{itemize}
\item[$(a)$] $g_i$ is homogeneous and $t^{b(1)}g_1f_1+\ldots
  +t^{b(s)}g_sf_{s}=0$;
\item[$(b)$] $q_if_i=g_i-b_iy_i^{b_i-1}$. 
\end{itemize}
In particular,
  $b_iy_i^{b_i-1}t^{b(i)}f_i\in (f_1,\ldots ,f_{i-1},f_{i+1},\ldots
  ,f_s)+I^2$.
\end{proposition}
\begin{proof}
Applying \cite[Remark~3.4]{opPCB} to $\widetilde{L}$, the PCB matrix
associated to $L$, one gets the syzygy $t^{b(1)}\widetilde{f}_1+\ldots
+t^{b(s)}\widetilde{f}_{s}=0$, where $\widetilde{f}_{i}$ is the binomial
defined by the $i$-th column of $\widetilde{L}$. Note that
$\widetilde{f}_{i}=f_{\widetilde{l}_{*,i}}=x_i^{b_i}-y_i^{b_i}=
(x_i-y_i)(x_i^{b_{i}-1}+\cdots +x_i^{b_i-j}y_i^j+\cdots +y_i^{b_i-1})=
f_ig_i$ (even if $b_i=1$). It follows that $t^{b(1)}g_1f_1+\ldots
+t^{b(s)}g_sf_{s}=0$. Since $\mathbf{d} L=0$ for some
$\mathbf{d}\in\mathbb{N}_+^s$ (see Theorem~\ref{GPCB-new}), 
clearly $\mathbf{d}\widetilde{L}=0$ and
$\widetilde{f}_{i}$ is homogeneous. Since $\widetilde{f}_i=f_ig_i$ and $S$ is
a domain, $g_i$ is homogeneous too. This proves $(a)$. On the other
hand, one has the following identity:
\begin{eqnarray*}
&&q_if_i=(x_i^{b_i-2}+2x_i^{b_i-3}y_i+\cdots
  +(b_i-1)y_i^{b_i-2})(x_i-y_i)=\\&&x_i^{b_i-1}+x_i^{b_i-2}y_i+\cdots
  +x_iy_i^{b_i-2}-(b_i-1)y_i^{b_i-1}=g_i-b_iy_i^{b_i-1}.
\end{eqnarray*}
To see where this equality comes from, consider the polynomials
$f=X-Y$ and $g=X^{b-1}+X^{b-2}Y+\cdots +Y^{b-1}$ in a polynomial ring
$K[X,Y]$, where $b$ is a positive integer. Set $Z=X/Y$ and
dehomogenize $f$ and $g$ to obtain $u=Z-1$ 
and $v=Z^{b-1}+Z^{b-2}+\cdots +1$ in the Euclidean domain $K[Z]$. If
${\rm char}(K)=0$ or ${\rm char}(K)=p$, $p$ a prime with $p\nmid b$, $u$ and $v$ are
relatively prime. The Euclidean Algorithm explicitly gives us two
polynomials $\alpha,\beta\in K[Z]$ with $\alpha u+\beta v=1$. On
rehomogenizing and multiplying by $b$ one gets the desired identity,
which holds in any characteristic. This proves $(b)$. Finally, on
multiplying the equality $q_if_i=g_i-b_iy_i^{b_i-1}$ by $t^{b(i)}f_i$,
one gets $b_iy_i^{b_i-1}t^{b(i)}f_i=t^{b(i)}g_if_i-q_it^{b(i)}f_i^2\in
(f_1,\ldots ,f_{i-1},f_{i+1},\ldots ,f_s)+I^2$.
\end{proof}

\begin{lemma}\label{gradeI} Let $I(L)=(f_1,\ldots ,f_s)$ be the
ideal of a GPCB matrix $L$. Then, any subset of $s-1$ elements of $f_1,\ldots ,f_s$ is a
  regular sequence in $S$.
\end{lemma}

\begin{proof} The ideal $I(L)$ is graded by
Theorem~\ref{GPCB-new}. Thus the lemma follows using 
the proof of \cite[Proposition~3.3]{opPCB}. 
\end{proof}

The following result generalizes \cite[Proposition~3.5]{opPCB}.

\begin{corollary}\label{aci} 
 Let $I=I(L)$ be the ideal of a GPCB matrix $L$. Then the following hold.
\begin{itemize}
\item[$(a)$] For any associated prime $\mathfrak{p}$ of $I$, either
  ${\rm ht}(\mathfrak{p})=s-1$ and $t_{i}\not\in\mathfrak{p}$, for all
  $i=1,\ldots ,s$, or else $\mathfrak{p}=\mathfrak{m}$.
\item[$(b)$] For any minimal prime ideal $\mathfrak{p}$ over $I$,
  $IS_{\mathfrak{p}}$ is a complete intersection.
\item[$(c)$] If $s\geq 3$, $I$ is an almost complete intersection.
\end{itemize}
\end{corollary}
\begin{proof}
Item $(a)$ follows from Remark~\ref{onecanapply},
Theorem~\ref{GPCB-new} and Lemma~\ref{sep16-12}. Item (c) follows from
Proposition~\ref{feb9-13}. To prove $(b)$, let $I=(f_1,\ldots,f_s)$,
let $\mathfrak{p}$ be an arbitrary minimal prime over $I$, so
$t_i\not\in\mathfrak{p}$ for all $i=1,\ldots ,s$. Either ${\rm
  char}(K)=0$, or else we may suppose that ${\rm char}(K)=p$, $p$ a
prime with $p\nmid b_s$, because $\gcd({\bf b})=1$ (and $s\geq 2$).
In particular, by Proposition~\ref{explicitrelation},
$b_sy_s^{b_s-1}\not\in\mathfrak{p}$ (otherwise, if we had
$b_sy_s^{b_s-1}\in\mathfrak{p}$, it would follow that $y_s$ is in
$\mathfrak{p}$, so one of $t_1,\ldots ,t_{s-1}$ is in $\mathfrak{p}$,
a contradiction) and
$g_s=q_sf_s+b_sy_s^{b_s-1}\not\in\mathfrak{p}$. Therefore
$t^{b(s)}g_{s}\notin \frak{p}$ and it follows from
Proposition~\ref{explicitrelation}(a) and Lemma~\ref{gradeI} that
$IS_{\mathfrak{p}}=(f_1,\ldots ,f_{s-1})S_{\mathfrak{p}}$ is generated
by a regular sequence in $S_{\mathfrak{p}}$.
\end{proof}

In the next pair of results, we give an explicit description of the
hull of a GPCB ideal and, if $s\geq 4$, of an irredundant embedded
component.

\begin{proposition}\label{explicitui} 
$($cf. \cite[Proposition~4.4]{opPCB}$)$ Let $I=I(L)=(f_1,\ldots
,f_s)$ be the ideal of a GPCB matrix $L$. Set $J=(f_1,\ldots
  ,f_{s-1})$. Suppose that $g\in (J\colon f_s)$ is such that $g\not\in\mathfrak{p}$
  for any minimal prime $\mathfrak{p}$ over $I$. Then the following
  hold.
\begin{itemize}
\item[$(a)$] $I(\mathcal{L})={\rm Hull}(I)=(I\colon g)=(J\colon g)$.
\item[$(b)$] For $b(s)\in\mathbb{N}^{s}$ and $g_s\in S$ as in
  Proposition~\ref{explicitrelation},
  $I(\mathcal{L})=(I\colon t^{b(s)}g_s)=(J\colon t^{b(s)}g_s)$.
\end{itemize}
\end{proposition}
\begin{proof}
Since $J$ is a graded complete intersection and $f_s\not\in J$,
$(J\colon f_s)\subset\mathfrak{m}$ and $g\in\mathfrak{m}$. 
By Proposition~\ref{structureI(L)}, $I(\mathcal{L})=\
$Hull$(I)=(I\colon g^{\infty })$.  
Hence $I\subset (J\colon g)\subset (I\colon g)\subset
(I\colon g^{\infty })=I(\mathcal{L})$. To
finish the proof of $(a)$, just proceed as in
\cite[Proposition~4.4]{opPCB}. By Proposition~\ref{explicitrelation},
$t^{b(s)}g_s\in (J\colon f_s)$ and
$t^{b(s)}g_s\in\mathfrak{m}\setminus\mathfrak{p}$, for all 
minimal prime ideals $\mathfrak{p}$ over $I$. Thus (b) follows from $(a)$.
\end{proof}

\begin{theorem}\label{embeddedcomp} $($cf. \cite[Theorem~4.10]{opPCB}$)$
Let $s\geq 4$. Let $I=I(L)=(f_1,\ldots ,f_s)$ be the ideal of a GPCB
matrix $L$. Suppose that $(I\colon g)=(I\colon g^{\infty})$ for some
$g\in \mathfrak{m}$, $g\not\in\mathfrak{p}$ for any minimal prime
ideal $\mathfrak{p}$ over 
$I$. Then the following hold.
\begin{itemize}
\item[$(a)$] $I+(g)$ is an irredundant $\mathfrak{m}$-primary
  component of $I$;
\item[$(b)$] For $b(s)\in\mathbb{N}^s$ and $g_s\in S$ as in
  Proposition~\ref{explicitrelation}, $I+(x^{b(s)}g_s)$ is an
  irredundant $\mathfrak{m}$-primary component of $I$.
\end{itemize}
\end{theorem}
\begin{proof}
This is an straightforward extension of \cite[Theorem~4.10]{opPCB}.
\end{proof}

In the light of the preceding results, and taking into account 
Theorem~\ref{main-new} and Proposition~\ref{structureI(L)}, it is not
hard to see that analogues of \cite[Theorem~6.5 and
Theorem~7.1]{opPCB} hold for 
GPCB ideals.
The details are left to the
interested reader. 

For $s=2$, we now give an explicit description of a GPCB ideal and its
hull in terms of the entries of the corresponding GPCB matrix. This
description will be used later in Section~\ref{lattice-dim1-3vars}
(see Proposition~\ref{when-is-unmixed}).

\begin{lemma}\label{cases=2}
Let $I=I(L)=(f_1,f_2)$ be a GPCB ideal associated to a $2\times 2$
GPCB matrix $L$ and let $\mathcal{L}$ be the lattice of $\mathbb{Z}^2$
spanned by the columns of $L$. Then there exists
$(c_1,c_2)\in\mathbb{N}^2_+$ such that
$I(\mathcal{L})=(t_1^{c_1}-t_2^{c_2})$.  Moreover, either $I$ is a PCB
ideal, $I$ is principal and $I=I(\mathcal{L})$ is a lattice ideal, or
else $I$ is not a PCB ideal, $I$ is not principal and $I$ is not a
lattice ideal.
\end{lemma}
\begin{proof}
Set ${\bf b}=(b_1,b_2)\in \mathbb{N}^{2}_{+}$, $\gcd(\mathbf{b})=1$, 
with $L{\bf
  b}^{\top}=0$. Then $a_{1,1}b_1=a_{1,2}b_2$ and
$a_{2,1}b_1=a_{2,2}b_2$. Since $\gcd(\mathbf{b})=1$, this forces
$a_{1,1}=b_2c_1$, $a_{1,2}=b_1c_1$, for some $c_1\in \mathbb{N}_{+}$, and
$a_{2,1}=b_2c_2$, $a_{2,2}=b_1c_2$, for some
$c_2\in\mathbb{N}_{+}$. Therefore,
\begin{eqnarray*}
L=\left(
\begin{array}{rr}a_{1,1}&-a_{1,2}\\-a_{2,1}&a_{2,2}\end{array}\right)=
\left(
\begin{array}{rr}b_2c_{1}&-b_1c_{1}\\-b_2c_{2}&b_1c_{2}\end{array}\right).
\end{eqnarray*}
Set $h:=t_1^{c_1}-t_2^{c_2}$ and $g_i:=t_1^{(b_i-1)c_1}+\cdots +
t_1^{(b_i-j)c_1}t_{2}^{jc_2}+\cdots +t_2^{(b_i-1)c_2}$, for $i=1,2$
(if $b_i=1$, we understand that $g_i=1$). Then
$f_1=t_1^{b_2c_1}-t_2^{b_2c_2}=hg_2$ and
$f_2=t_2^{b_1c_2}-t_1^{b_1c_1}=-hg_1$. Hence $I=I(L)=(hg_2,hg_1)$ and
$I\subset(h)$. Since $\gcd(\mathbf{b})=1$, there exist $v_1,v_2\in\mathbb{Z}$
such that $1=v_1b_1+v_2b_2$. Hence
\begin{eqnarray*}
v_2(a_{1,1},-a_{2,1})-v_1(-a_{1,2},a_{2,2})=
v_2(b_2c_1,-b_2c_2)-v_1(-b_1c_1,b_1c_2)=(c_1,-c_2),
\end{eqnarray*}
and $\mathcal{L}=\langle (a_{1,1},-a_{2,1}),(-a_{1,2},a_{2,2})\rangle =\langle
(c_1,-c_2)\rangle$. Therefore $I=I(L)\subset 
I(\mathcal{L})=(t_1^{c_1}-t_2^{c_2})=(h)$.

If $b_1=1$ or $b_2=1$, then $g_1=1$ or $g_2=1$ and
$I=(hg_1,hg_2)=(h)=I(\mathcal{L})$. Moreover $I$ is the PCB ideal associated
to the PCB matrix with columns $(c_1,-c_2)^{\top}$ and
$(-c_1,c_2)^{\top}$.

Suppose that $b_1,b_2>1$, with $\gcd(\mathbf{b})=1$. Then $b_1\nmid b_2$
and $b_2\nmid b_1$. It follows that $f_2\nmid f_1$ and $f_1\nmid f_2$
(see, e.g., \cite[Lemma~8.2]{opHN}). Since $I=(f_1,f_2)$ is
homogeneous, $f_1,f_2$ is a minimal homogeneous system of generators of
$I$. Hence $I$ is not principal and $h\in I(\mathcal{L})\setminus I$. In
particular, $I$ is not a PCB ideal (by \cite[Remark~2.3]{opPCB}) and
is not a lattice ideal (see Proposition~\ref{structureI(L)}(b)).
\end{proof}

We finish the section with an example. 

\begin{example}\label{twoforthepriceofone}
Let $L$ be the following PCB matrix and let $I=I(L)$ and
$I^{\top}=I(L^{\top})$ 
be the matrix ideals of $L$
and $L^{\top},$ respectively. 
\begin{eqnarray*}
L=\left(\begin{array}{rrrr}
4& -2& -1& -1\\ 
-1& 4& -2& -1\\
-1& -1& 3& -1\\
-1&-1&-1&3
\end{array}\right).
\end{eqnarray*}
Observe that 
$L\mathbf{1}^\top=0$ and $L^{\top}\mathbf{b}^\top=0$, with 
$\mathbf{b}=(20,24,31,25)$ and $\mathbf{1}=(1,1,1,1)$. 
Let $\bl$ and $\bl^{\top}$ be the lattices spanned by
the columns of $L$ and $L^{\top}$, respectively. Here
$\Delta_3(L)=1$. By Corollary~\ref{degreeI(L)},
$\deg(S/I)=31$ and $\deg(S/I^{\top})=1$. Moreover 
$I=\mathfrak{p}_{\mathbf{b}}\cap \fq$ and
$I^{\top}=\mathfrak{p}_{\mathbf{1}}\cap
\fq^{\prime}$ are minimal primary decompositions of $I$ and
$I^{\top}$, respectively, where $I(\bl)=\mathfrak{p}_{\mathbf{b}}$ is
the Herzog 
ideal associated to $\mathbf{b}$, $I(\bl^{\top})=\mathfrak{p}_{\mathbf{1}}$ is
the Herzog ideal associated to $\mathbf{1}$, and $\fq$ and
$\fq^{\prime}$ are $\mathfrak{m}$-primary ideals which can be calculated
explicitly (see Theorem~\ref{embeddedcomp}).
\end{example}

\section{Laplacian matrices and ideals}\label{laplacian-section}

In this section we show how our results can be applied to an
interesting family of binomial ideals arising from Laplacian matrices. 

Connecting combinatorial properties of
graphs to linear-algebraic properties of Laplacian matrices 
has attracted a great deal of attention 
\cite{Biggs,godsil,laplacian-survey}. We are interested in 
relating the combinatorics of the graph with the algebraic
invariants and properties of the binomial ideals associated to
Laplacian matrices.   

Let $S=K[t_1,\ldots,t_s]$ be a polynomial ring over a field $K$ and
let $G =(V,E,w)$ be a weighted connected simple graph, where
$V=\{t_1,\ldots,t_s\}$ is the set of vertices, $E$ is the
set of edges and 
$w$ is a weight function that associates a weight $w_e$ with every
edge $e$ in the graph. Edges of $G$ are unordered pairs $\{t_i,t_j\}$
with $i\neq j$. To define the Laplacian matrix, recall that
the {\it adjacency matrix\/} $A(G)$ of this graph is given by 
$$
A(G)_{i,j}:=\left\{
\begin{array}{ll}
w_e&\ \mbox{ if } e=\{t_i,t_j\}\in E,\\
0&\ \mbox{ otherwise}.
\end{array}
\right.
$$ Now, the {\it Laplacian\/} $L(G)$ of the graph $G$ is defined as
$L(G) := D(G)-A(G)$, where $D(G)$ is a diagonal matrix with entry
$D(G)_{i,i}$ equal to the weighted degree $\sum_{e\in E(t_i)}w_e$ of
the vertex $t_i$. Here, we denoted by $E(t_i)$ the set of edges
adjacent (incident) to $t_i$. One can check that the entries of the
Laplacian are given by
$$
L(G)_{i,j}:=\left\{
\begin{array}{ll}
\sum_{e\in E(t_i)} w_e&\mbox{ if } i=j,\\
-w_e&\mbox{ if }i\neq j\mbox{ and }e=\{t_i,t_j\}\in E,\\ 
0&\ \mbox{otherwise}.
\end{array}
\right.
$$
Notice that $L(G)$ is symmetric and $\mathbf{1}L(G)=0$. 
The Laplacian matrix is a prime example of a CB 
matrix. The
Laplacian matrices of complete graphs are PCB matrices; this type of
matrix occurs in \cite{riemann-roch}.  The binomial ideal $I\subset
S$ defined by the columns of $L(G)$ is called the {\it Laplacian
ideal} of $G$. If $I\subset S$ is the Laplacian ideal of $G$, the
lattice ideal  
$I(\mathcal{L})=(I\colon(t_1\cdots t_s)^\infty)$ is called the 
{\it toppling\/} ideal of the graph \cite{riemann-roch,perkinson}. If
$G$ is 
connected, the toppling
ideal is a lattice ideal of dimension $1$. 

The torsion subgroup of the factor group $\mathbb{Z}^s/{\rm
  Im}(L(G))$, denoted by $K(G)$, is called the {\it critical group\/}
or the {\it sandpile group} of $G$ (see
\cite{alfaro-valencia,lorenzini} for additional information). Notice
that $K(G)$ is the torsion subgroup of $\mathbb{Z}^s/\mathcal{L}$. The
structure, as a finite abelian group, of $K(G)$ is only known for a
few families of graphs (see \cite{alfaro-valencia} and the references
there). If $G$ is regarded as a multigraph (where each edge $e$ occurs
$w_e$ times), then by the Kirchhoff's matrix tree theorem, the order
of $K(G)$ is the number of spanning trees of $G$ and this number is
equal to the $(i,j)$-entry of the adjoint matrix of $L(G)$ for any
$(i,j)$ (see \cite[Theorem~6.3, p.~39]{Biggs} and
\cite[Theorem~1.1]{laplacian-survey}).

Next we give an application of our earlier results to this setting.

\begin{proposition}\label{laplacian-appl} Let $G=(V,E,w)$ be a
connected weighted simple graph with 
vertices $t_1,\ldots,t_s$ and let $I\subset S$ be its Laplacian
ideal. Then the following hold.
\begin{itemize}
\item[(a)] $V(I,t_i)=\{0\}$ for all $i$.
\item[(b)] $\deg(S/I)=\deg(S/I(\mathcal{L}))=|K(G)|$.  
\item[(c)] ${\rm Hull}(I)=I(\mathcal{L})$.
\item[(d)] If $|E(t_i)|\geq 3$ for all $i$, then $I$ is not a lattice ideal.
\item[(e)] If $G=\mathcal{K}_s$ is a complete graph, then
$\deg(S/I)=s^{s-2}$.   
\item[(f)] If $G$ is a tree, then $\deg(S/I)=\deg(S/I(\mathcal{L}))=1$.
\end{itemize}
\end{proposition}

\begin{proof} (a): For $1\leq k\leq s$, let $f_k$ be the binomial
defined by the $k$-th column of $L(G)$. Fix $i$ such that $1\leq i\leq s$.
Clearly $\{0\}$ is contained in $V(I,t_i)$ because $I$ is graded. To
show the 
reverse containment, let $\alpha=(\alpha_1,\ldots,\alpha_s)$ 
be a point in $V(I,t_i)$.
If $\{t_i,t_k\}\in
E(t_i)$, we claim that $\alpha_k=0$. We
can write 
$$ 
f_k=t_k^{\sum_{e\in E(t_k)}w_e}-\prod_{e\in E(t_k),\, t_j\in
e}t_j^{w_e}.
$$
Using that $f_k(\alpha)=0$, we get 
$\alpha_k^{\sum_{e\in E(t_k)}w_e}=\prod_{\{e\in E(t_k),\, t_j\in
e\}}\alpha_j^{w_e}$. Since $\{t_i,t_k\}$ is in $E(t_k)$ 
and using that $\alpha_i=0$, we obtain that $\alpha_k=0$, as claimed.
Let $\ell$ be an integer in $\{1,\ldots,s\}$. Since the graph $G$ is
connected, there is a path $\{v_1,\ldots,v_r\}$ joining $t_i$ and
$t_\ell$, i.e., $v_1=t_i$, $v_r=t_\ell$ and 
$\{v_j,v_{j+1}\}\in E(G)$ for all $j$. There is a permutation $\pi$
of $V$ such that $\pi(1)=i$, $\pi(r)=\ell$ and $v_j=t_{\pi(j)}$ for
$j=1,\ldots,r$. Then $t_{\pi(j)}\in 
E(t_{\pi(j-1)})$ for $j=2,\ldots,r$. Applying the claim successively 
for $j=2,\ldots,r$, we obtain
$\alpha_{\pi(2)}=0,\alpha_{\pi(3)}=0,\ldots,\alpha_{\pi(r)}=0$. 
Thus,  $\alpha_\ell=0$. This proves that $\alpha=0$.  

(b) and (c): These two parts follow from
Proposition~\ref{additivity-of-the-degree}, 
Corollary \ref{nov28-12} and Proposition~\ref{structureI(L)}(a), 
since $L(G)$ 
is homogeneous with respect to the weight vector $\mathbf{1}$.

(d): This part follows from Proposition~\ref{oct6-12}(a).

(e): By part (b) one has $\deg(S/I)=|K(G)|$, and by
\cite[p.~39]{Biggs} one has $|K(G)|=s^{s-2}$. 

(f): Since $|K(G)|$ is the
number of spanning trees of the graph $G$, 
this number is equal to $1$. Thus,
$K(G)=\{0\}$, i.e., $K(G)$ is torsion free. 
By (b), we get that $\deg(S/I)=1$. 
\end{proof}

\begin{example}\label{weighted-graph-example} 
Let $G$ be the weighted graph of Figure~\ref{weighted-graph} and let $I$ be
its Laplacian ideal. 
Then $I=(t_1^3-t_2t_3^2,\, t_2^8-t_1t_3^3t_4^4,\,
t_3^6-t_1^2t_2^3t_4,\, t_4^5-t_2^4t_3)$, $\deg(S/I)=67$ and 
$$
{\rm Hull}(I)=I(\mathcal{L})=(t_1^3-t_2t_3^2,\,t_1t_3^4-t_2^4t_4,\,
t_2^4t_3-t_4^5,\, t_3^6-t_1^2t_2^3t_4,\, t_2^8-t_1t_3^3t_4^4,\,
t_1^2t_2^7-t_3^5t_4^4).
$$
If $K=\mathbb{Q}$, the toppling ideal $I(\mathcal{L})$ has two
primary components of degrees $66$ and $1$.

\begin{figure}[h]
\setlength{\unitlength}{.025cm}
\thicklines
\begin{picture}(300,75)
\put(0,0){\circle*{5.5}}
\put(10,-5){$t_4$}
\put(-35,35){\circle*{5.5}}
\put(-35,35){\line(1,0){70}}
\put(-55,35){$t_2$}
\put(0,70){\circle*{5.5}}
\put(6,70){$t_1$}
\put(-25,55){\tiny $1$}
\put(23,55){\tiny $2$}
\put(0,40){\tiny $3$}
\put(-25,10){\tiny $4$}
\put(23,10){\tiny $1$}
\put(35,35){\circle*{5.5}}
\put(41,35){$t_3$}
\put(0,0){\line(1,1){35}}
\put(0,0){\line(-1,1){35}}
\put(-35,35){\line(1,1){35}}
\put(0,70){\line(1,-1){35}}
\put(100,30){$L(G)=\left(\begin{array}{rrrr}3&-1&-2&0\\
-1&8&-3&-4\\
-2&-3&6&-1\\
0&-4&-1&5
\end{array}\right)$}
\end{picture}
\caption{Weighted graph $G$.}
\label{weighted-graph}
\end{figure}
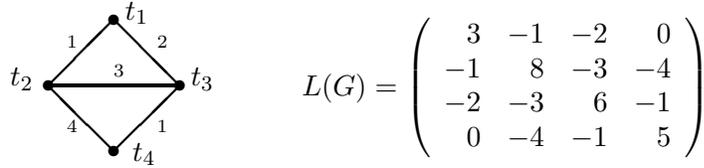

\end{example}

As another application, by Proposition~\ref{feb9-13}, the Laplacian
ideal is an almost complete intersection for any connected simple
graph without vertices of degree $1$.

\begin{proposition}\label{aci-laplacian} Let $G=(V,E,w)$ be a 
connected weighted simple graph with vertices $t_1,\ldots,t_s$ and let
$I\subset S$ be its Laplacian ideal. If $|E(t_i)|\geq 2$ for all $i$,
then $I$ is an almost complete intersection.
\end{proposition}

The notion of a Laplacian matrix can be extended to weighted 
digraphs, see \cite{corrales-valencia} and the references there. 
Let $G=(V,E,w)$ be a weighted 
digraph without loops and with vertices
$t_1,\ldots,t_s$, let $w(t_i,t_j)$ be the weight of the directed arc
from $t_i$ to $t_j$ and let $A(G)$ be the adjacency matrix of $G$
given by $A(G)_{i,j}=w(t_i,t_j)$. The Laplacian matrix
of $G$ is given by $L(G)=D^+(G)-A(G)$, where $D^+(G)$ is the diagonal
matrix with the out-degrees of the vertices of $G$ in the diagonal
entries. Note that $L(G)\mathbf{1}^\top=0$ and that the Laplacian matrix 
of a digraph may not be symmetric (see Example~\ref{march5-13}). If
$t_i$ is a sink, i.e., there is no arc of the form 
$(t_i,t_j)$, then the $i$-th row of $L(G)$ is zero. Thus, the rank of
$L(G)$ may be much less than $s-1$. When $G$ is a strongly connected
digraph, it is well-known that the rank of $L(G)$ is $s-1$. This
follows from the Perron-Frobenius theorem; see the proof of
Theorem~\ref{GCB-new-perron-Frobenius} .  

Let $G$ be a weighted digraph without sources or sinks, i.e., for each
vertex $t_i$ there is at least one arc of the form $(t_j,t_i)$ and
one arc of 
the form $(t_i,t_k)$. Then $L(G)$ is a CB matrix. Conversely if $L$ is
a CB matrix as in Definition~\ref{square-matrix-new}, then $L$ is the
Laplacian matrix of the weighted digraph $G$ defined as follows. A
pair $(t_i,t_j)$ is an arc of $G$ if and only if $i\neq j$ and
$a_{i,j}\neq 0$. The weight of the arc $(t_i,t_j)$ is $a_{i,j}$. 

\begin{definition}{\cite[p.~175,
p.~29]{godsil}}\label{underlying-digraph} Let $A=(a_{i,j})$ be an
$s\times s$ real matrix. The  
{\it underlying digraph\/} of $A$, denoted by $G_A$, has vertex set
$\{t_1,\ldots,t_s\}$, with an arc from vertex $t_i$ to vertex $t_j$ if
and only if $a_{i,j}\neq 0$. Note that this digraph may have loops. 
A digraph is called {\it strongly connected\/} if any two
vertices can be joined by a directed path. 
\end{definition} 

The underlying digraph of
the matrix of Example~\ref{weighted-graph-example} is strongly connected. 
If $G$ is a weighted
digraph and $L$ is its Laplacian matrix, then $G$ is obtained from
the underlying digraph $G_L$ of $L$ by removing all loops of $G_L$.

\begin{theorem}\label{GCB-new-perron-Frobenius} 
Let $L$ be a GCB matrix and let $G$ be its underlying graph. If $G$ is
strongly connected, then ${\rm rank}(L)=s-1$ and $L^\top$ is a GCB
matrix.
\end{theorem}

\begin{proof} By passing to the associated CB matrix $\widetilde{L}$, we
may assume that $L$ is a CB matrix (see the proof of
Theorem~\ref{GPCB-new}(c)). Let $L$ be a CB matrix as in
Definition~\ref{square-matrix-new}. We can write $L=D-A$, where
$D=\diag(a_{1,1},\ldots,a_{s,s})$ and $A$ is the matrix whose $i,j$
entry is $a_{i,j}$ if $i\neq j$ and whose diagonal entries are equal
to zero. We set $\delta_i=a_{i,i}$ for $i=1,\ldots,s$. By 
hypothesis $L\mathbf{1}^\top=0$, hence ${\rm rank}(L)\leq s-1$. There
exists a nonzero vector $\mathbf{b}\in\mathbb{Z}^s$ such that
$\mathbf{b}L=0$. Therefore 
$$ \mathbf{b}D=\mathbf{b}A=\mathbf{b}D(D^{-1}A).$$ Since
$L\mathbf{1}^\top=0$, we get that $D\mathbf{1}^\top=A\mathbf{1}^\top$
or equivalently $(D^{-1}A)\mathbf{1}^\top=\mathbf{1}^\top$. Thus, as
the entries of $D^{-1}A$ are nonnegative, the matrix $B:=D^{-1}A$ is
stochastic. It is well-known that the spectral radius $\rho(B)$ of a
stochastic matrix $B$ is equal to $1$
\cite[Theorem~5.3]{nonnegative-matrices}, where $\rho(B)$ is the
maximum of the moduli of the eigenvalues of $B$. As the diagonal
entries of $B$ are zero and $\delta_i>0$ for all $i$, the underlying
digraph $G_B$ of $B$ is equal to the digraph obtained from $G$ by
removing all loops of $G$. Since $G$ is strongly connected so is
$G_B$, and by the Perron-Frobenius Theorem for nonnegative matrices
\cite[Theorem~8.8.1, p.~178]{godsil}, $\rho(B)=1$ and $1$ is a simple
eigenvalue of $B$ (i.e., the eigenspace of $B$ relative to $\rho(B)=1$
is $1$-dimensional), and if $z$ is an eigenvector for $\rho(B)=1$,
then no entries of $z$ are zero and all have the same sign. Applying
this to $z=\mathbf{b}D$, we get that $b_i\neq 0$ for all $i$ and all
entries of $\mathbf{b}$ have the same sign. Hence, ${\rm
  ker}(L^{\top})=(\mathbf{b}^\top)$ for any non-zero vector
$\mathbf{b}$ such that $\mathbf{b}L=0$, so $L^\top$ is a GCB matrix of
rank $s-1$.
\end{proof}

The results of the previous sections can also be applied to GCB 
ideals that arise from matrices with strongly connected underlying
digraphs.

\begin{proposition}\label{gcb-with-strongly-connected-digraph} 
Let $L$ be a GCB matrix of size $s\times s$, let $G_L$ be the
underlying digraph of $L$ and let $I=I(L^\top)$ be the
matrix ideal of $L^\top$. The following conditions are equivalent.
\begin{itemize}
\item[\rm(a)] $G_L$ is strongly connected. 
\item[\rm(b)] $V(I,t_i)=\{0\}$ for all $i$.
\item[\rm(c)] $L_{i,j}>0$ for all $i,j$, where ${\rm
adj}(L)=(L_{i,j})$ is the 
adjoint of $L$.
\end{itemize}

\end{proposition}

\begin{proof} By passing to the associated matrix $\widetilde{L}$, we
may assume that $L$ is a CB matrix (see the proof of
Theorem~\ref{GPCB-new}(c)). 

(a) $\Rightarrow$ (b): Since any two
vertices can be joined by a directed path, the 
proof follows adapting the argument given to prove
Proposition~\ref{laplacian-appl}(a). 

(b) $\Rightarrow$ (a): We proceed by contradiction. Assume that $G_L$ is
not strongly connected. Without loss of generality we may assume that
there is no directed path from $t_1$ to $t_s$. Let $W$ be the set of
all vertices $t_i$ such that there is a directed path from $t_i$ to
$t_s$, the vertex $t_s$ being included in $W$. The set $W$ is nonempty
because $G_L$ has no sources or sinks by definition of a CB matrix,
and 
the vertex $t_1$ is not in $W$. Consider the vector $\alpha\in K^s$ 
defined as $\alpha_i=1$ if $t_i\notin W$ and $\alpha_i=0$ if $t_i\in
W$. To derive a contradiction it suffices to show that all binomials
of $I(L^\top)$ vanish at the nonzero vector $\alpha$. Let $L$ be
as in Definition~\ref{square-matrix-new} and let 
$f_k=t_k^{a_{k,k}}-\prod_{j\neq k}t_j^{a_{k,j}}$ be the binomial defined
by the $k$-row of $L$. If $t_k\in W$, there is a directed path
$\mathcal{P}$ from $t_k$ to $t_s$. Then $t_j$ is part of the path
$\mathcal{P}$ for some $j$ such that $a_{k,j}>0$. Thus, 
since $t_j\in W$, $f_k(\alpha)=0$. If $t_k\notin W$, then $t_j$ is 
not in $W$ for any $j$ such that $a_{k,j}>0$, because if $a_{k,j}>0$,
the pair $(t_k,t_j)$ is an arc of $G_L$. Thus, $f_k(\alpha)=0$. 

(a) $\Rightarrow$ (c): By the proof of Theorem~\ref{GPCB-new}(c)),
one has that 
$L_{i,i}\geq 0$ for all $i$. By
Theorem~\ref{GCB-new-perron-Frobenius} and using the proof of 
Theorem~\ref{GPCB-new}(a), we get that 
$L_{i,j}>0$ for all $i,j$. 

(c) $\Rightarrow$ (a):  We proceed by contradiction. Assume that $G_L$ is
not strongly connected. We may assume that there is no directed path
from $t_1$ to $t_s$. Let $W$ be as above and let $W^c=\{t_i\vert\,
t_i\notin W\}$ be its complement. We can write
$W^c=\{t_{\ell_1},\ldots,t_{\ell_r}\}$. Consider the $r\times r$ submatrix
$B$ obtained from $L$ by fixing rows $\ell_1,\ldots,\ell_r$ and columns
$\ell_1,\ldots,\ell_r$. Notice that by the arguments above 
$W^c=\cup_{t_k\in W^c}({\rm supp}(f_k))$. Hence any row of $B$ extends
to a row of $L$ by adding $0$'s only, and consequently the sum of the columns
of $B$ is zero. Hence $\det(B)=0$. By permuting rows and columns, $L$
can be brought to the form
$$
L'=\left(\begin{array}{cc}
C&0 \\ 
C'&C''
\end{array}\right)
$$
where $C$ and $C''$ are square matrices of orders $r$ and $s-r$, 
respectively, and $\det(C)=0$. Hence the adjoint of $L'$ has a zero
entry, and so does the adjoint of $L$, a contradiction. 
\end{proof}

\begin{corollary}\label{gcb-with-strongly-connected-digraph-cor} 
Let $L$ be a GCB matrix of size $s\times s$ and let $(L_{i,j})$ be its
adjoint. If $G_L$ is strongly connected, then
$$
\gcd(\{L_{i,k}\}_{i=1}^s)\deg(S/I(L^\top))=
\max(\{L_{i,k}\}_{i=1}^s)\gcd(\{L_{i,j}\})\ \mbox{ for any }\ k.
$$
\end{corollary}

\begin{proof} By
Proposition~\ref{gcb-with-strongly-connected-digraph}, all entries of
${\rm adj}(L)$ are positive and any column of ${\rm adj}(L)$ 
gives a grading for $I(L^\top)$. Hence the formula follows from 
Corollary~\ref{degreeI(L)}.
\end{proof}

\begin{example}\label{march5-13} Let $G$ be the weighted digraph of
Figure~\ref{fig3}  
and let $L$ be its Laplacian matrix. The digraph $G_L$ is not
strongly connected, the CB ideal $I(L^\top)$ is 
graded but $I(L)$ is not.

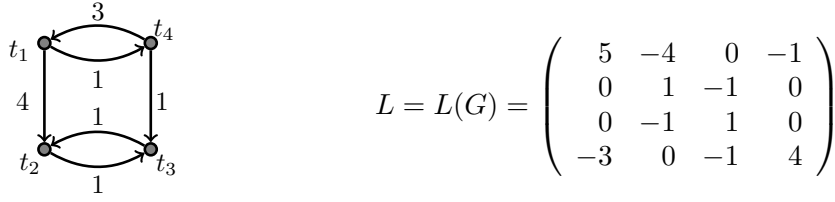
\begin{figure}[h]
\setlength{\unitlength}{.025cm}
\thicklines
\begin{picture}(400,85)(0,-30)
\begin{tabular}{c@{\extracolsep{2cm}}c}
\multirow{5}{3cm}{
	\begin{tikzpicture}[line width=1pt, scale=1]
		\tikzstyle{every node}=[inner sep=0pt, minimum width=4.5pt] 
		\draw (135:1)+(0,0.18) node (v1) [draw, circle, fill=gray] {};
		\draw (225:1)+(0,0.18) node (v2) [draw, circle, fill=gray] {};
		\draw (315:1)+(0,0.18) node (v3) [draw, circle, fill=gray] {};
		\draw (45:1)+(0,0.18) node (v4) [draw, circle, fill=gray] {};
		\draw (150:1.20)+(0,0.18) node () {\small $t_1$};
		\draw (225:1.27)+(0,0.18) node () {\small $t_2$};
		\draw (315:1.3)+(0,0.18) node () {\small $t_3$};
		\draw (45:1.25)+(0,0.18) node () {\small $t_4$};
		\draw (45:1)+(-0.7,0.6) node () {\small $3$};
		\draw (45:1)+(-0.7,-1.7) node () {\small $1$};
		\draw (45:1)+(-0.7,-0.8) node () {\small $1$};
		\draw (45:1)+(-0.7,-0.3) node () {\small $1$};
		\draw (45:1)+(0.15,-0.6) node () {\small $1$};
		\draw (45:1)+(-1.7,-0.6) node () {\small $4$};
		\draw (v1) edge[->] (v2);
		\draw[->] (0,0) {(v2) to [bend right] (v3)};
		\draw[->] (0,0) {(v3) to [bend right] (v2)};
		\draw (v4) edge[->] (v3);
		\draw[->] (0,0) {(v1) to [bend right] (v4)};
		\draw[->] (0,0) {(v4) to [bend right] (v1)};
	\end{tikzpicture}
}
& \\
&
$
L=L(G)=
\left(\begin{array}{rrrr}
 5 & -4 &  0 &  -1\\
 0 &  1 & -1 & 0\\
 0 & -1 &  1 & 0\\
-3 &  0 &  -1 &  4
\end{array}\right)
$\\
& \\
\end{tabular}
\end{picture}
\caption{A weighted digraph $G$ with four vertices and its Laplacian matrix.}
\label{fig3}
\end{figure}
\end{example}

\section{Homogeneous lattice ideals of dimension $1$ in $3$ variables
}\label{lattice-dim1-3vars}

The main results of this section
uncover the structure of lattice ideals of dimension $1$ in $3$
variables and the structure of the
homogeneous lattices of rank $2$ in $\mathbb{Z}^3$.

\begin{definition}\label{fullset} Let $I$ be a binomial ideal of $S$ and
let $f=t^a-t^b\in I$. We will say that: 
\begin{itemize}
\item[$(a)$] $f$ is $t_i$-{\em pure} 
if $t^a$ and $t^b$ are non-constant, have no common variables and 
${\rm supp}(a)=\{ i\}$; 

\item[$(b)$] $f$ is $t_i$-{\em critical} if $f$ is $t_i$-pure and for
  any other $t_i$-pure binomial $g=t_i^{c_i}-t^d$ of $I$, $c_i\geq
  a_i$; 

\item[$(c)$] A {\em full set} of pure (respectively, critical)
  binomials of $I$ is a family $f_1,\ldots ,f_s$ of binomials where
  each $f_i$ is a $t_i$-pure (respectively, $t_i$-critical) binomial
  of $I$.
\end{itemize}
\end{definition}

We begin with a result that complements the well-known result of Herzog
\cite{He3} that shows that the toric ideal of a monomial space curve is
generated by a full set of critical binomials.

\begin{theorem}\label{nov18-12}
Let $S=K[t_1,t_2,t_3]$ and let $I$ be a homogeneous lattice ideal of
$S$ of height $2$. Then $I$ is generated by a full set of critical
binomials. Concretely, and with a suitable renumbering of the
variables, only the two following cases can occur:
\begin{itemize}
\item[$(a)$] $I$ is minimally generated by $f_1=t_1^{a_{1}}-t_3^{c_3}$
  and $f_2=t_2^{b_2}-t_1^{b_1}t_3^{b_3}$, with $0\leq b_1\leq a_1$,
  $a_1,b_2,c_3>0$ and $b_1+b_3>0$;
\item[$(b)$] $I$ is minimally generated by
  $f_1=t_1^{a_{1}}-t_2^{a_2}t_3^{a_3}$,
  $f_2=t_2^{b_2}-t_1^{b_1}t_3^{b_3}$ and
  $f_3=t_3^{c_3}-t_1^{c_1}t_2^{c_2}$, with $0<a_2<b_2$, $0<a_3<c_3$,
  $0<b_1<a_1$, $0<b_3<c_3$, $0<c_1<a_1$ and $0<c_2<b_2$. Moreover,
  $a_1=b_1+c_1$, $b_2=a_2+c_2$ and $c_3=a_3+b_3$.
\end{itemize}
\end{theorem}
\begin{proof} 
Let $\mathbf{d}=(d_1,d_2,d_3)$, $\gcd(\mathbf{d})=1$, be the grading in $S$
under which $I$ is homogeneous. Since $I$ is a lattice ideal, $t_i$ is
a non-zero divisor of $S/I$ for $i=1,2,3$. In particular, $I$ can be
generated by pure binomials, i.e., binomials of the form
$t_1^{e_1}-t_2^{e_2}t_3^{e_3}$, with $e_1>0$, and similarly for
$i=2,3$.

Since $I$ is a homogeneous lattice ideal of height $2$, by
\cite[Proposition~2.9]{ci-lattice}, $V(I,t_i)=\{0\}$ for all $i$. In
particular, $I$ contains $t_i$-pure binomials, for $i=1,2,3$. Indeed,
if $I$ contains no $t_3$-pure binomials, say, then $V(I,t_1)\supset
V(t_1,t_2)\neq \{0\}$, a contradiction.

Therefore there exist $f_1=t_1^{a_{1}}-t_2^{a_{2}}t_{3}^{a_{3}}$,
$f_2=t_2^{b_{2}}-t_1^{b_{1}}t_{3}^{b_{3}}$ and
$f_3=t_3^{c_{3}}-t_1^{c_{1}}t_{2}^{c_{2}}$, a full set of critical
binomials of $I$; i.e., $a_1,b_2,c_3>0$ and for any $t_1$-pure
binomial of $I$ of the form $t_1^{e_1}-t_2^{e_2}t_3^{e_3}$, one has
$e_{1}\geq a_{1}$, and similarly with the other variables $t_2$ and
$t_3$. Notice that one could have $f_j=-f_i$ for $j\neq i$.

Following the proof of Kunz in \cite[pp.~137--140]{kunz}, one can show
that $I$ is generated by a full set of critical binomials. For the
sake of clarity, we outline the main details of the proof.

After renumbering the variables one may suppose that $f_1$ is the one
of least degree among $f_1$, $f_2$ and $f_3$. Then $a_2\leq b_2$ and
$a_3\leq c_3$. Moreover, $a_2=b_2$ is equivalent to $a_{3}=0$, and, in
this case, $-f_1=t_2^{b_{2}}-t_1^{a_{1}}$ is $t_2$-critical and one
may choose $f_2$ to be $-f_1$. Similarly, $a_3=c_3$ is equivalent to
$a_2=0$, and, in this case, $-f_1=t_3^{c_{3}}-t_1^{a_{1}}$ is
$t_3$-critical and one may choose $f_3$ to be $-f_1$.

If $a_2=b_2$, (i.e., $a_3=0$), we may interchange the numbering of
the variables 
$t_2$ and $t_3$ so that, in the new numbering, $a_3=c_3$ and $a_2=0$.
Hence, there are only the  following two cases for $f_1$:
\begin{itemize}
\item[$(a)$:] $f_1=t_1^{a_{1}}-t_{3}^{c_{3}}$, if $a_2=0$ and
  $a_3=c_3$;
\item[$(b)$:] $f_1=t_1^{a_{1}}-t_{2}^{a_{2}}t_{3}^{a_{3}}$, with
  $0<a_2<b_2$ and $0<a_3<c_3$.
\end{itemize}

\noindent \underline{\sc Case $(a)$}. Here
$f_1=t_1^{a_{1}}-t_{3}^{c_{3}}$, $f_2=t_2^{b_2}-t_1^{b_1}t_3^{b_3}$
and $f_3=t_3^{c_3}-t_1^{a_1}=-f_1$, with $\deg(f_1)\leq
\deg(f_2)$. Moreover, one can suppose that $0\leq b_1\leq a_1$.

In a rather long, but not difficult way, one proves that for any pure
binomial $f$ of $I$, either $f$ is in $(f_1,f_2)$ or  $f$ modulo
$(f_1,f_2)$ is a multiple of a binomial $g$ of $I$ with
$\deg(g)<\deg(f)$. One starts by taking
$f=t_1^{e_1}-t_2^{e_2}t_3^{e_3}\in I$, with $e_1\geq a_1$, and proving
that $f-t_1^{e_1-a_1}f_1$ is a multiple of a binomial $g\in I$ with
$\deg(g)<\deg(f)$.  There is an analogous argument if
$f=t_2^{e_2}-t_1^{e_1}t_3^{e_3}\in I$ or
$f=t_3^{e_3}-t_1^{e_1}t_2^{e_2}\in I$.

One concludes that $I=(f_1,f_2)$. Indeed, suppose not and take $f$ a
pure binomial in $I\setminus (f_1,f_2)$ of the smallest possible
degree. We have seen that there exists $h\in (f_1,f_2)$ such that
$f-h$ is a multiple of a binomial $g\in I$ with
$\deg(g)<\deg(f)$. Since $f$ is the element in $I\setminus (f_1,f_2)$
of the smallest possible degree, this forces $g$ to be in
$(f_1,f_2)$. Hence $f\in (g,h)\subset (f_1,f_2)$, a contradiction.

\noindent \underline{\sc Case $(b)$}. Here
$f_1=t_1^{a_{1}}-t_2^{a_{2}}t_{3}^{a_{3}}$,
$f_2=t_2^{b_2}-t_1^{b_1}t_3^{b_3}$ and
$f_3=t_3^{c_3}-t_1^{c_1}t_2^{c_2}$, with $0<a_2<b_2$ and $0<a_3<c_3$.

Possibly after renumbering $t_2$ and $t_3$, one can assume that the
degree of $f_2$ is smaller than the degree of $f_3$.  Observe that
$b_1<a_1$ and, in particular, $b_3>0$. Analogously, $c_1<a_1$ and, in
particular, $c_2>0$. Moreover $b_3\leq c_3$, and $b_3=c_3$ is
equivalent to $b_1=0$. In this case, $-f_2=t_3^{c_3}-t_2^{b_2}$ is
$t_3$-critical and one may choose as $f_3$ the binomial $-f_2$.

Therefore, there are only the following two cases for $f_2$:
\begin{itemize}
\item[$(b.1)$:] $f_2=t_2^{b_2}-t_3^{c_3}$, if $b_1=0$ and $b_3=c_3$;
\item[$(b.2)$:] $f_2=t_2^{b_2}-t_1^{b_1}t_3^{b_3}$, with $0<b_1<a_1$
  and $0<b_3<c_3$.
\end{itemize}

\noindent \underline{\sc Case $(b.1)$}. Here
$f_1=t_1^{a_{1}}-t_2^{a_{2}}t_{3}^{a_{3}}$, $f_2=t_2^{b_2}-t_3^{c_3}$
and $f_3=t_3^{c_3}-t_2^{b_2}=-f_2$, with $0<a_2<b_2$, $0<a_3<c_3$
and $\deg(f_1)\leq \deg(f_2)$. Similarly to the proof of Case $(a)$,
one can show that $I=(f_1,f_2)$ (although now the doubly-pure binomial
has bigger degree than the other binomial).

\noindent \underline{\sc Case $(b.2)$}. Here
$f_1=t_1^{a_{1}}-t_2^{a_{2}}t_{3}^{a_{3}}$,
$f_2=t_2^{b_2}-t_1^{b_1}t_3^{b_3}$ and
$f_3=t_3^{c_3}-t_1^{c_1}t_2^{c_2}$, with $0<a_2<b_2$, $0<a_3<c_3$, $0<
b_1<a_1$ and $0<b_3<c_3$. Moreover, $0\leq c_1<a_1$, $c_2>0$ and
$\deg(f_1)\leq \deg(f_2)\leq \deg (f_3)$. Since $b_3>0$, then
$c_2<b_2$. Since $b_3<c_3$ (and $b_1\neq 0$), then $c_1>0$.

As in Case $(a)$, one can prove that, for each pure binomial $f$ of
$I$, either $f$ is in $(f_1,f_2,f_3)$ or $f$ modulo $(f_1,f_2,f_3)$ is
a multiple of a binomial $g$ of $I$ with $\deg(g)<\deg(f)$. One
concludes, as before, that $I=(f_1,f_2,f_3)$.

In Case $(b.2)$, $I$ is minimally generated by $f_1$, $f_2$ and $f_3$.
Indeed, if $f_3\in (f_1,f_2)$, say, then on taking $t_1=0$ and
$t_2=0$, one would get a contradiction.

Finally, $a_1=b_1+c_1$, $b_2=a_2+c_2$ and $c_3=a_3+b_3$ (see, here,
\cite[p~139, line~15]{kunz}). Indeed, let $\alpha_1:=a_1-b_1-c_1$,
$\alpha_2:=b_2-a_2-c_2$ and $\alpha_3:=c_3-a_3-b_3$. Clearly
$\alpha_1d_1+\alpha_2d_2+\alpha_3d_3=0$. We may suppose that
$\alpha_2$ and $\alpha_3$, say, have the same sign. Then necessarily
$\alpha_1=0$, because if not, since $f_1$ is $t_1$-critical and either
$t_1^{\alpha_1}-t_2^{-\alpha_2}t_3^{-\alpha_3}\in I$ or
$t_1^{-\alpha_1}-t_2^{\alpha_2}t_3^{\alpha_3}\in I$, we get that
either $\alpha_1\geq a_1$ or $-\alpha_1\geq a_1$, which would imply
that either $-b_1-c_1\geq 0$ or $b_1+c_1\geq 2a_1$, in a contradiction
to $0<b_1<a_1$ and $0<c_1<a_1$. Thus $\alpha_1=0$, so $\alpha_2=0$
and $\alpha_3=0$.
\end{proof}
\begin{corollary} 
If $I$ is a lattice ideal of dimension $0$ and $s=2$, then $I$ is
generated by at most $3$ binomials.
\end{corollary}
\begin{proof} 
By Lemma~\ref{o'carroll-comments}, $I^h\subset S[u]$ is a graded
lattice ideal of dimension $1$. Thus, the result follows from
Theorem~\ref{nov18-12}.
\end{proof}
Before proceeding with the main result of the section, we state some
properties of GCB and CB binomial ideals in the case $s=3$.
\begin{lemma}\label{GCBishom}
Let $I=I(L)$ be the GCB ideal associated to a $3\times 3$ GCB matrix
$L$. Then $I$ is homogeneous and $V(I,t_i)=\{0\}$ for all $i$.
\end{lemma}
\begin{proof}
Let $L$ be a $3\times 3$ GCB matrix, where ${\bf
  b}=(b_1,b_2,b_3)\in\bn^3_+$, $\gcd({\bf b})=1$ and $L{\bf
  b}^{\top}=0$,
\begin{eqnarray*}
L=\left(\begin{array}{rrr}
a_{1,1}&-a_{1,2}&-a_{1,3}\\-a_{2,1}&a_{2,2}&-a_{2,3}\\-a_{3,1}&-a_{3,2}&a_{3,3}
\end{array}\right).
\end{eqnarray*}
If $L$ is a GPCB matrix, $I=I(L)$ is homogeneous and $V(I,t_i)=\{0\}$
for all $i$ (see Remark~\ref{onecanapply}). Suppose that $L$ is not a
GPCB matrix, 
but a GCB matrix. After renumbering the variables, one may suppose that
$a_{2,1}=0$. In particular, $a_{3,1}>0$, $a_{1,2}+a_{3,2}>0$ and
$b_2a_{2,2}=b_3a_{2,3}$, so $a_{2,3}>0$. Let $h_{*,3}$ be the third
row of $\adj(L)$, the adjoint matrix of $L$:
\begin{eqnarray*}
h_{*,3}=(a_{2,2}a_{3,1},a_{1,1}a_{3,2}+a_{1,2}a_{3,1},a_{1,1}a_{2,2}).
\end{eqnarray*}
It follows that $h_{*,3}\in\bn^3_+$. Set
$\mathbf{d}=h_{*,3}/\gcd(h_{*,3})$. Then $\mathbf{d}\in\bn^3_+$,
$\gcd(\mathbf{d})=1$ and 
$\mathbf{d} L=0$. Hence $I=I(L)$ is homogeneous. Moreover,
\begin{eqnarray*}
I=(t_1^{a_{1,1}}-t_3^{a_{3,1}},t_2^{a_{2,2}}-t_1^{a_{1,2}}t_3^{a_{3,2}},
t_3^{a_{3,3}}-t_1^{a_{1,3}}t_2^{a_{2,3}}).
\end{eqnarray*}
Clearly $\rad(I,t_1)=\mathfrak{m}$ and $\rad(I,t_3)=\mathfrak{m}$.
Since $a_{2,3}>0$, it 
follows that $\rad(I,t_2)=\mathfrak{m}$ too. Thus $V(I,t_i)=\{0\}$ for all $i$.
\end{proof}

\begin{proposition}\label{CBproperties}
Let $I=I(L)$ be the CB ideal associated to a $3\times 3$ CB matrix
$L$.  Then the following conditions hold.
\begin{itemize}
\item[$(a)$] $t_2^{a_{2,3}}f_1+t_3^{a_{3,1}}f_2+t_1^{a_{1,2}}f_3=0$
  and $t_3^{a_{3,2}}f_1+t_1^{a_{1,3}}f_2+t_2^{a_{2,1}}f_3=0$;
\item[$(b)$] For $\{ i,j,k\}=\{1,2,3\}$, then $f_i,f_j,t_k$ is a
  regular sequence (in any order);
\item[$(c)$] $I$ is either a complete intersection or an almost
  complete intersection;
\item[$(d)$] $I$ is an unmixed ideal of height $2$;
\item[$(e)$] $I$ is a homogeneous lattice ideal and
  $I=I(\bl)=(I:(t_1t_2t_3)^{\infty})=\hull(I)$.
\end{itemize}
\end{proposition}
\begin{proof}
The proof of $(a)$ follows from a simple check. By
Lemma~\ref{GCBishom}, $S$ can be graded, with $t_1,t_2,t_3$
and $f_1,f_2,f_3$ homogeneous elements of positive degree. Using
\cite[Proposition~4.2, $(c)$]{opHN}, one deduces $(b)$. If $I$ is a
PCB ideal, by Corollary~\ref{aci}, $I$ is an almost complete
intersection of height $2$. If $I$ is not a PCB ideal, $a_{i,j}=0$, 
for some $i\neq j$. Using $(a)$, $I$ is generated
by two of the three $f_1,f_2,f_3$. In particular, by $(b)$, $I$ is a
complete intersection. This proves $(c)$. If $I$ is a PCB ideal, by
\cite[Remark~4.7 and Proposition~3.3]{opPCB}, $I$ is an unmixed ideal
of height $2$. If $I$ is not a PCB ideal, $a_{i,j}=0$, 
for some $i\neq j$, by $(c)$, $I$ is a complete intersection, hence
unmixed too. This proves
$(d)$. By Lemma~\ref{GCBishom}, $I=I(L)$ is graded and
$V(I,t_i)=\{0\}$ for all $i$. Moreover $I=I(L)$ is an unmixed binomial ideal
associated to an integer matrix $L$ (in fact a CB matrix). By
Proposition~\ref{structureI(L)}, $I=I(\bl)=(I:(t_1t_2t_3)^{\infty})=\hull(I)$ is
a homogeneous lattice ideal.
\end{proof}

As a consequence of Theorem~\ref{nov18-12} and
Proposition~\ref{CBproperties}, we obtain the main result of the
section.

\begin{theorem}\label{mainSection}
Let $S=K[t_1,t_2,t_3]$ and let $I$ be an ideal of $S$.  Then $I$ is a
homogeneous lattice ideal of dimension $1$ if and only if $I$ is a CB
ideal.
\end{theorem}
\begin{proof} Suppose that $I$ is a homogeneous lattice ideal of $S$ 
of dimension $1$. By Theorem~\ref{nov18-12}, only two cases can occur
for $I$. In the first case, $I=(f_1,f_2)$, where
$f_1=t_1^{a_{1}}-t_3^{c_3}$ and $f_2=t_2^{b_2}-t_1^{b_1}t_3^{b_3}$,
with $0\leq b_1\leq a_1$ and $a_1,b_2,c_3>0$. Then $I$ is the CB ideal
associated to the CB matrix $L$, where
\begin{eqnarray*}
L=\left( \begin{array}{rrc}
  a_1&-b_1&-a_1+b_1\phantom{+}\\0&b_2&-b_2\\-c_3&-b_3&b_3+c_3
\end{array}\right),
\end{eqnarray*}
because $f_3=-t_3^{b_3}f_1-t_1^{a_1-b_1}f_2\in (f_1,f_2)$.

In the second case, $I=(f_1,f_2,f_3)$, where
$f_1=t_1^{a_{1}}-t_2^{a_2}t_3^{a_3}$,
$f_2=t_2^{b_2}-t_1^{b_1}t_3^{b_3}$ and
$f_3=t_3^{c_3}-t_1^{c_1}t_2^{c_2}$, with $0<a_2<b_2$, $0<a_3<c_3$,
$0<b_1<a_1$, $0<b_3<c_3$, $0<c_1<a_1$ and $0<c_2<b_2$, with
$a_1=b_1+c_1$, $b_2=a_2+c_2$ and $c_3=a_3+b_3$. Then $I$ is the PCB
ideal associated to the PCB matrix $L$, where
\begin{eqnarray*}
L=\left( \begin{array}{rrr}a_1&-b_1&-c_1\\-a_2&b_2&-c_2\\-a_3&-b_3&c_3
\end{array}\right).
\end{eqnarray*}
Conversely, if $I$ is a CB ideal, $I$ is a graded lattice ideal
of dimension $1$ by Proposition~\ref{CBproperties}.
\end{proof}

Next, we show that the homogeneous lattices of rank $2$ in
$\mathbb{Z}^3$ are precisely the lattices generated by the columns of
a CB matrix.

\begin{corollary}\label{jan6-13} 
Let $\bl$ be a lattice of rank $2$ in $\bz^3$. Then, $\bl$ is
homogeneous if and only if $\bl$ is generated by the columns of a CB
matrix.
\end{corollary}
\begin{proof} 
If $\bl$ is homogeneous of rank $2$, $I(\bl)$ is homogeneous of height
$2$. By Theorem~\ref{mainSection}, $I(\bl)$ is a CB ideal. Hence, by
Corollary~\ref{nov4-12}, $\bl$ is generated by the columns of a CB
matrix. Conversely, let $\bl$ be the lattice generated by the columns
of a $3\times 3$ CB matrix $L$. Clearly $L$ has rank $2$ and, by
Proposition~\ref{CBproperties}, $\mathbf{d} L=0$ for some
$\mathbf{d}\in\bn^{3}_{+}$. In particular, $\bl$ is homogeneous.
\end{proof}

In the next result we add a new condition for a GPCB ideal to be a
lattice ideal (see Proposition~\ref{2variables} and
Proposition~\ref{structureI(L)} $(b)$).

\begin{proposition}\label{when-is-unmixed} 
Let $I=I(L)$ be the binomial ideal associated to an $s\times s$ GPCB
matrix $L$. Then $I$ is a lattice ideal if and only if $s\leq 3$ and
$I$ is a PCB ideal.
\end{proposition}

\begin{proof} $\Rightarrow$) Assume that $I$ is a lattice ideal. 
By Proposition~\ref{oct6-12}, $s\leq 3$.
If $s=2$, by Lemma~\ref{cases=2}, $I$ is a PCB ideal. If $s=3$, by
Theorem~\ref{GPCB-new}, $I$ is a graded lattice ideal. 
Hence, by Proposition~\ref{feb9-13}, $I$ cannot be a complete
intersection. Applying Theorem~\ref{nov18-12}, we get as
in the final paragraph of the proof of Theorem~\ref{mainSection} that
$I$ is a PCB ideal.

$\Leftarrow)$ Assume that $I$ is a PCB ideal. In particular $I$ is a
CB ideal. If $s=3$, by Theorem~\ref{mainSection}, $I$ is a lattice
ideal. If $s=2$, by Lemma~\ref{cases=2}, $I$ is a lattice ideal. 
\end{proof}

As a corollary of Theorem~\ref{mainSection} we deduce the structure of
the hull of a GCB ideal. 

\begin{corollary}\label{hull(GPCB)}
Let $I=I(L)$ be the GCB ideal associated to a $3\times 3$ GCB matrix
$L$. Then $I(\bl)$ is a CB ideal.
\end{corollary}
\begin{proof}
By Lemma~\ref{GCBishom}, $I$ is homogeneous with
$V(I,t_i)=\{0\}$ for all $i$. Therefore, $I(\bl)$ is a homogeneous
lattice ideal of 
dimension 1 (see Proposition~\ref{structureI(L)}). Thus, by
Theorem~\ref{mainSection}, $I(\bl)$ is a CB ideal.
\end{proof}

We deduce a method to find a generating set for the hull of a GCB
ideal.

\begin{procedure}\label{findhull}
Given a $3\times 3$ GCB matrix $L$,
\begin{itemize}
\item[$(a)$] Find a CB matrix $M$ such that $\bm=\bl$, where $\bm$ and
  $\bl$ are the lattices of $\bz^3$ spanned by the columns of $M$ and
  $L$, respectively. Equivalently, find a CB matrix $M$ and a $3\times
  3$ integer matrix $Q$ with $\det(Q)=1$, such that $LQ=M$.
\item[$(b)$] By Proposition~\ref{CBproperties}, $I(M)$ is a lattice
  ideal and $I(M)=I(\bm)$. Hence $I(\bl)=I(M)$.
\end{itemize}
\end{procedure}

We illustrate this method with some examples.

\begin{example}\label{hullisPCB}
Let $I=(f_1,f_2,f_3)=(t_1^4-t_2t_3,t_2^3-t_1^5t_3,t_3^3-t_1^3t_2)$ be
the GPCB ideal associated to the GPCB matrix
\begin{eqnarray*}
L=\left(\begin{array}{rrr}
4&-5&-3\\-1&3&-1\\-1&-1&3
\end{array}\right).
\end{eqnarray*}
Here $L\mathbf{b}^\top=0$, with ${\bf b}=(2,1,1)$. We have
$(-1,2,-2)=(4,-1,-1)+(-5,3,-1)\in\bl$. Therefore $\bl=\langle
(4,-1,-1),(-1,2,-2),(-3,-1,3)\rangle$. Take
\begin{eqnarray*}
M=\left(\begin{array}{rrr}
4&-1&-3\\-1&2&-1\\-1&-2&3
\end{array}\right),
\end{eqnarray*}
which is a PCB matrix. Hence $I(\bl)=I(\bm)=I(M)=
(t_1^4-t_2t_3,t_2^2-t_1t_3^2,t_3^3-t_1^3t_2)$. In this example, the
hull of a GPCB ideal is a PCB ideal.
\end{example}

\begin{example}\label{hullisnotPCB}
Let $I=(f_1,f_2,f_3)=(t_1^4-t_2^2t_3,t_2^3-t_1t_3,t_3-t_1t_2)$ be the
GPCB ideal associated to the GPCB matrix
\begin{eqnarray*}
L=\left(\begin{array}{rrr}
4&-1&-1\\-2&3&-1\\-1&-1&1
\end{array}\right).
\end{eqnarray*}
Here $L\mathbf{b}^\top=0$, with ${\bf b}=(2,3,5)$. Let $Q$ and $M$ be
the $3\times 3$ integer 
matrices:
\begin{eqnarray*}
Q=\left(\begin{array}{rrr}
1&1&0\\1&2&0\\2&2&1
\end{array}\right)\mbox{, }\det(Q)=1\mbox{; }
M=\left(\begin{array}{rrr} 1&0&-1\\-1&2&-1\\0&-1&1
\end{array}\right),\mbox{ a CB matrix; }LQ=M.
\end{eqnarray*}
Hence $I(\bl)=I(\bm)=I(M)= (t_1-t_2,t_2^2-t_3,t_3-t_1t_2)$.  In this
example, the hull of a GPCB ideal is a CB ideal.
\end{example}

\begin{example}\label{tocompletetoPCB}
Let $\mathcal{L}=\langle(-2,4,-2),(-2,-3,4)\rangle$, which is a rank 2
homogeneous lattice with respect to the vector $(5,6,7)$. Thus the
lattice ideal $I(\mathcal{L})$ of $\mathcal{L}$ is a graded lattice
ideal of dimension 1. By Theorem~\ref{nov18-12}, $I(\mathcal{L})$ is
generated by a full set of critical binomials and, by
Theorem~\ref{mainSection}, $I(\mathcal{L})$ is a CB ideal (here a PCB
ideal). Concretely
\begin{eqnarray*}
I(\mathcal{L})=((t_2^4-t_1^2t_3^2,t_1^2t_2^3-t_3^4)\colon(t_1t_2t_3)^
\infty)=(t_1^4-t_2t_3^2,t_2^4-t_1^2t_3^2,t_1^2t_2^3-t_3^4).
\end{eqnarray*}
To obtain the above generating set one may ``complete'' the two
generators of $\bl$ to a CB (in fact a PCB) matrix $M$, namely,
\begin{eqnarray*}
M=\left(\begin{array}{rrr}
4&-2&-2\\-1&4&-3\\-2&-2&4
\end{array}\right),
\end{eqnarray*}
and apply Procedure~\ref{findhull}. The degree of $S/I$ is $14$. If
$K=\mathbb{Q}$, then $I=\mathfrak{p}_1\cap\mathfrak{p}_2$, where
$\mathfrak{p}_1,\mathfrak{p}_2$ are prime ideals of degree $7$.
\end{example}

\begin{example}\label{jan10-13}
Let $\mathcal{L}=\langle(2,-1,-1),(-3,1,-1)\rangle$, which is a rank 2
non-homogeneous lattice. The lattice ideal $I(\mathcal{L})$ of
$\mathcal{L}$ is a non-graded lattice ideal of height 2. By
Theorem~\ref{mainSection}, $I(\mathcal{L})$ cannot be a CB
ideal. Concretely
\begin{eqnarray*}
I(\mathcal{L})=((t_1^2-t_2t_3,t_2-t_1^3t_3)\colon(t_1t_2t_3)^
\infty)=(t_1^2-t_2t_3,t_1t_3^2-1)
\end{eqnarray*}
If we apply  Corollary~\ref{inspired-by-francesc-example} with
$v_1=-2$, $v_2=-5$ and $v_3=1$, we get that $I(\mathcal{L})$ has
degree $6$. 
\end{example}

\medskip

\noindent
{\bf Acknowledgments.} 
The authors would like to thank the referees for their
careful reading of the paper and for the improvements that
they suggested.

\bibliographystyle{plain}

\end{document}